\documentclass[11pt,a4paper]{article}
\usepackage{amssymb,amsmath,amsthm,epsfig,verbatim,pstricks,url}
\usepackage{graphbox} % align graphics on top
\usepackage{enumitem} % set space above itemize and enumerate
\usepackage{ifpdf}  % use \pdfoutput=1 to force pdflatex output
\newtheorem{theorem}{\sc Theorem.}[section]
\newtheorem{lemma}[theorem]{\sc Lemma.}
\newtheorem{remark}[theorem]{\sc Remark.}
\newtheorem{example}[theorem]{\sc Example.}
\renewcommand{\theequation}{\arabic{section}.\arabic{equation}}
\newenvironment{AMS}%
{{\upshape\bfseries AMS subject classifications. }\ignorespaces}{}
\newenvironment{keywords}{{\upshape\bfseries Key words. }\ignorespaces}{}

\newcommand{\bRplus}{{\mathbb R}_{>0}}
\newcommand{\bRgeq}{{\mathbb R}_{\geq 0}}

\newcommand{\RZ}{{\mathbb R} \slash {\mathbb Z}}
\newcommand{\bR}{{\mathbb R}}
\newcommand{\bC}{{\mathbb C}}

\newcommand{\bS}{{\mathbb S}}
\newcommand{\bN}{{\mathbb N}}

\newcommand{\altm}{a}
\newcommand{\altg}{{\mathfrak g}}
\newcommand{\altbeta}{\widehat\beta}

\newcommand{\tr}{\operatorname{tr}}

\renewcommand{\div}{\operatorname{div}}

\newcommand{\diag}{\operatorname{diag}}

\newcommand{\ratio}{{\mathfrak r}}
\newcommand{\dH}[1]{\;{\rm d}{\mathcal{H}}^{#1}} % Hausdorff measure
\newcommand{\drho}{\;{\rm d}\rho}
\newcommand{\dt}{\;{\rm d}t}
\newcommand{\ds}{\;{\rm d}s}
\newcommand{\dtheta}{\;{\rm d}\theta}
\newcommand{\Vh}{\underline{V}^h}

\newcommand{\Id}{{\rm Id}}

\newcommand{\dd}[1]{\frac{\rm d}{{\rm d}#1}}
\newcommand{\ddt}{\dd{t}}
\newcommand{\Dds}{\frac{D}{{\rm d}\tilde s}}

\def\epsilon{\varepsilon} 

\def\hat{\widehat}

\begin{document}
\title{%\texttt{\jobname} \quad {\rm \mydate} \\
A novel finite element approximation of anisotropic curve shortening flow}
\author{Klaus Deckelnick\footnotemark[2]\ \and 
        Robert N\"urnberg\footnotemark[3]}

\renewcommand{\thefootnote}{\fnsymbol{footnote}}
\footnotetext[2]{Institut f\"ur Analysis und Numerik,
Otto-von-Guericke-Universit\"at Magdeburg, 39106 Magdeburg, Germany \\
{\tt klaus.deckelnick@ovgu.de}}
\footnotetext[3]{Dipartimento di Mathematica, Universit\`a di Trento,
38123 Trento, Italy \\ {\tt robert.nurnberg@unitn.it}}

\date{}

\maketitle

\begin{abstract}
We extend the DeTurck trick from the classical isotropic curve shortening flow
to the anisotropic setting. Here the anisotropic
energy density is allowed to depend on space, which allows an interpretation
in the context of Finsler metrics,
giving rise to e.g.\ geodesic curvature flow in Riemannian manifolds.
Assuming that the density is strictly convex and smooth, we introduce a novel
weak formulation for anisotropic curve shortening flow.
We then derive an optimal $H^1$--error bound for a 
continuous-in-time semidiscrete finite element approximation that uses
piecewise linear elements.
In addition, we consider some fully practical fully discrete schemes 
and prove their unconditional stability.
Finally, we present several numerical simulations, including some
convergence experiments that confirm the derived error bound,
as well as applications to crystalline curvature flow and
geodesic curvature flow.
\end{abstract} 


\renewcommand{\thefootnote}{\arabic{footnote}}

\setcounter{equation}{0}
\section{Introduction} % \label{sec:intro}

The aim of this paper is to introduce and analyze a novel approach to 
approximate the evolution of curves by anisotropic curve shortening flow. 
The evolution law that we consider arises as a natural gradient flow
for the anisotropic, spatially inhomogeneous, energy
\begin{equation} \label{eq:aniso}
\mathcal E(\Gamma) = \int_\Gamma \altm(z)\gamma(z,\nu) \dH1(z)
= \int_\Gamma \altm \gamma(\cdot,\nu) \dH1
\end{equation}
for a closed curve $\Gamma$, with unit normal $\nu$, that is  
contained in the domain $\Omega \subset \bR^2$. In the above,
$\gamma: \Omega \times \bR^2 \rightarrow \bRgeq$ denotes the anisotropy 
function and $\altm : \Omega \rightarrow \bRplus$
is a positive weight function.  In the spatially homogeneous
case, i.e.\
\begin{equation} \label{eq:homogeneous}
\gamma(z,p)= \gamma_0(p) \quad \text{and} \quad \altm(z)= 1
\qquad \forall\ z \in \Omega = \bR^2,
\end{equation}
the corresponding functional $\mathcal E$  frequently occurs as an interfacial energy, e.g.\
in models of crystal growth, \cite{TaylorCH92,Gurtin93}. Our more general setting is
motivated by the work \cite{BellettiniP96} of Bellettini and Paolini, who consider the gradient flow for a perimeter
functional $P_\phi$ that is associated with a Finsler metric $\phi$. Equations (2.5), (2.6) in \cite{BellettiniP96}
show that $\mathcal E(\Gamma)=P_\phi(\Gamma)$ if one chooses $\gamma$ as the dual of $\phi$
and $\altm$ in terms of the $2$--dimensional $\phi$--volume. As an important special case we mention
that the choices
\begin{equation} \label{eq:riemgamma}
\gamma(z,p)= \sqrt{G^{-1}(z) p \cdot p} \quad \text{and} \quad 
\altm(z)=\sqrt{\det G(z)}
\end{equation}
can be used to describe the length of a curve in a two-dimensional Riemannian manifold $(\mathcal M, g)$. In this
case $G(z)$ is the first fundamental form arising from a local parameterization of $\mathcal M$, cf.\ Example \ref{ex:1}\ref{item:ex1c} below. 
Apart from being of geometric interest, the functional $\mathcal E$ also has applications in image processing, \cite{ClarenzDR03}. \\
The natural gradient flow for the energy $\mathcal E$ evolves a family of 
curves $\Gamma(t) \subset \Omega$ according to the law 
\begin{equation}  \label{eq:acsfintro}
\mathcal{V}_\gamma = \varkappa_\gamma, % \quad \mbox{ on } \Gamma(t).
\end{equation}
where $\mathcal{V}_\gamma$ and $\varkappa_\gamma$ are the anisotropic normal velocity and the anisotropic curvature, respectively.
The precise definitions of these quantities are based on a formula
for the first variation of $\mathcal E$, and they will be given in 
Section~\ref{sec:prelim}, below. In the isotropic case, i.e.\ when
\begin{equation} \label{eq:iso}
\gamma(z,p)=|p| \quad \text{and} \quad \altm(z)= 1
\qquad \forall\ z \in \Omega = \bR^2,
\end{equation}
we have that (\ref{eq:acsfintro}) is just the well-known curve shortening flow,
$\mathcal V = \varkappa$, with $\mathcal V$ and $\varkappa$ denoting the 
normal velocity and the curvature of $\Gamma(t)$, respectively.
For theoretical aspects of the anisotropic evolution law \eqref{eq:acsfintro}
we refer to \cite{Bellettini04a,Giga06}. Further information 
on (spatially homogeneous) anisotropic surface energies and the 
corresponding gradient flow can be found in  
\cite{DeckelnickDE05,Giga06} and the references therein.

In this paper we are interested in the numerical solution of (\ref{eq:acsfintro}) based on a parametric description of the evolving curves, i.e.\ $\Gamma(t)=
x(I,t)$ for some mapping $x:I \times [0,T] \rightarrow \Omega$. 
Here most of the existing literature has focused on the spatially homogeneous
case \eqref{eq:homogeneous}. 
Then the law \eqref{eq:acsfintro} reduces to
$\frac1{\gamma_0(\nu)} \mathcal V = \varkappa_{\gamma_0}$, see
Section~\ref{sec:prelim} for details, which can be viewed as a special case of
the weighted anisotropic curvature flow
\begin{equation} \label{eq:betaflow}
\altbeta_0(\nu) \mathcal V = \varkappa_{\gamma_0},
\end{equation}
for some mobility function $\altbeta_0$, 
see e.g.\ \cite[(8.20)]{DeckelnickDE05}. 
In \cite{Dziuk99}, a finite element scheme is proposed and analyzed for 
\eqref{eq:betaflow} with $\altbeta_0=1$. The method uses a variational formulation
of the parabolic system
\begin{equation} \label{eq:anisosystem}
x_t = \varkappa_{\gamma_0} \nu,
\end{equation}
and is generalized to higher codimension in \cite{Pozzi07}. 
A drawback of this approach is that the above system is degenerate in tangential direction and that the
numerical analysis requires 
an additional equation for the length element. A way to circumvent this difficulty consists in replacing (\ref{eq:anisosystem}) by a strictly parabolic system with
the help of a suitable tangential motion, known as DeTurck's trick in the literature. For the isotropic case, corresponding schemes have been suggested 
and analyzed in \cite{DeckelnickD95,ElliottF17}.
We mention that alternative parametric approaches for \eqref{eq:betaflow} 
allow for some benign tangential motion, see e.g.\
\cite{MikulaS01,MikulaS04a,triplejANI,fdfi}. \\
Since the choice \eqref{eq:riemgamma} allows to describe the length of a curve in a Riemannian manifold $(\mathcal M,g)$, it is
possible to use \eqref{eq:acsfintro} in order to treat geodesic curvature flow in $\mathcal M$ within our framework. 
Existing parametric approaches for this flow on a hypersurface of 
$\bR^3$ include the work \cite{MikulaS04} for the flow on a graph,
as well as \cite{curves3d} for the case that the hypersurface is given as a 
level set. Moreover, numerical schemes for the geodesic curvature flow 
(and other flows) of curves in a locally flat two-dimensional Riemannian 
manifold have been proposed in \cite{hypbol}. 
To the best of our knowledge, no error bounds have been derived for a numerical
approximation of geodesic curvature flow in Riemannian manifolds in the literature so
far.

In this paper we propose and analyze a new method for solving (\ref{eq:acsfintro}), which is based on DeTurck's trick, and which applies to general, spatially inhomogeneous anisotropies. Let
us outline the contents of the paper and  describe  the main results  of our work. 
By taking advantage of the fact that the function $(z,p) \mapsto \frac{1}{2} \altm^2(z) \gamma^2(z,p)$ is strictly convex in $p$, 
we derive in Section~\ref{sec:deturck} 
a strictly parabolic system whose solution satisfies (\ref{eq:acsfintro}). It turns out that this system can be written in a variational form, which makes it accessible to discretization by linear finite elements. In the isotropic case, the resulting numerical scheme is precisely the method proposed and analyzed in \cite{DeckelnickD95}, while in the anisotropic
case we obtain a novel scheme that can be considered as a generalization of the ideas in \cite{DeckelnickD95} and \cite{ElliottF17}. As one of the main results of this paper
 we show in Section~\ref{sec:fem} an optimal $H^1$--error bound in the continuous-in-time semidiscrete case. 
Unlike in \cite{Dziuk99} and \cite{Pozzi07}, the corresponding
proof does not need an equation for the length element because of the strict 
parabolicity of the underlying partial differential equation.
In order to discretize in time, we use the backward Euler method.
In particular, in Section~\ref{sec:fds}, as another important contribution
of our work, we introduce unconditionally stable
fully discrete finite element approximations for the following scenarios:
\renewcommand{\labelenumi}{(\alph{enumi})}
\renewcommand{\theenumi}{(\alph{enumi})}
\begin{enumerate}[noitemsep,topsep=-5pt]
\itemsep 0pt
\item
a spatially homogeneous, smooth anisotropy function 
$\gamma(z,p) = \gamma_0(p)$;
\item \label{item:BGN}
a spatially homogeneous anisotropy function 
$\gamma(z,p)=\sum_{\ell=1}^L \sqrt{\Lambda_{\ell} p \cdot  p}$, where $\Lambda_\ell$ are symmetric and positive definite matrices; 
\item 
a spatially inhomogeneous anisotropy function of the form \eqref{eq:riemgamma} 
to model geodesic curvature flow in a two-dimensional Riemannian manifold.
\end{enumerate}
In particular, the functions in \ref{item:BGN} can be used to approximate the case of a crystalline anisotropy, cf.\ \cite{triplejANI}.
Using these three fully discrete schemes, we present in Section~\ref{sec:nr} 
results of test calculations that confirm our error bound and show that the tangential motion that is introduced in our
approach has a positive effect on the distribution of grid points along the discrete curve. 

As our approach is based on a parameterization of the evolving curves, we only briefly mention numerical methods that employ an implicit description such as
the level-set method or the phase-field approach. The interested reader may consult \cite{ObermanOTT11,eck} for anisotropic curve shortening flow, as well
as \cite{ChoppS93,ChengBMO02,SpiraK07} for the geodesic curvature flow. These papers also provided additional references.

We end this section with a few comments about notation.
Throughout, we let $I=\RZ$ denote the periodic interval $[0,1]$.
We adopt the standard notation for Sobolev spaces, denoting the norm of
$W^{\ell,p}(I)$ ($\ell \in \bN_0$, $p \in [1, \infty]$)
by $\|\cdot \|_{\ell,p}$ and the 
seminorm by $|\cdot|_{\ell,p}$. 
For $p=2$, $W^{\ell,2}(I)$ will be denoted by
$H^{\ell}(I)$ with the associated norm and seminorm written as,
respectively, $\|\cdot\|_\ell$ and $|\cdot|_\ell$.
The above are naturally extended to vector functions, and we will write 
$[W^{\ell,p}(I)]^2$ for a vector function with two components.
For later use we recall the well-known Sobolev embedding
$H^1(I) \hookrightarrow C^0(I)$, i.e.\ there exists $C_I > 0$ such that
\begin{equation} \label{eq:sobolev}
\|f\|_{0,\infty} \leq C_I \|f\|_1 \qquad \forall\ f \in H^1(I).
\end{equation}
Furthermore, throughout the paper $C$ will denote a generic positive 
constant independent of the mesh parameter $h$, see below.
At times $\epsilon$ will play the role of a (small)
positive parameter, with $C_\epsilon>0$ depending on $\epsilon$, but
independent of $h$.
Finally, in this paper we make use of the Einstein summation convention.

\setcounter{equation}{0}
\section{Anisotropy and anisotropic curve shortening flow} \label{sec:prelim}

Let $\Omega \subset \bR^2$ be a domain,
and let $\altm \in C^2(\Omega, \bRplus)$. Moreover, we assume that 
$\gamma \in C^0(\Omega \times \bR^2, \bRgeq) \cap 
C^3(\Omega \times (\bR^2 \setminus \{ 0 \}), \bRplus)$,
as well as
\begin{equation} \label{eq:phi1} 
\gamma(z, \lambda p) =  \lambda  \gamma(z,p) \qquad \forall\
z \in \Omega,\ p \in \bR^2,\ \lambda \in \bRplus ,
\end{equation}
which means that $\gamma$ is positively one-homogeneous with respect to the
second variable.
It is not difficult to verify that (\ref{eq:phi1}) implies that
\begin{align}
& \gamma_p(z,\lambda p) = \gamma_p(z,p), \quad
\gamma_p(z,p)\cdot p = \gamma(z,p) 
\quad\mbox{and}\quad
\gamma_{pp}(z,p) p  = 0 \nonumber \\ & \hspace{8cm}
\qquad \forall\ z \in \Omega,\  p \in \bR^2\setminus\{0\},\  
\lambda \in \bRplus.
\label{eq:phidd}
\end{align}
Here $\gamma_p=(\gamma_{p_j})_{j=1}^2$ and 
$\gamma_{pp}=(\gamma_{p_i p_j})_{i,j=1}^2$ 
denote the first and second derivatives of $\gamma$ with respect to the second
argument. Similarly, we let $\gamma_z=(\gamma_{z_j})_{j=1}^2$ denote the 
derivatives of $\gamma$ with respect to the first argument.
We note for later use that on differentiating \eqref{eq:phidd} with respect to
$z$ we immediately obtain that the functions $\gamma_{z_j}(z,\cdot)$ and 
$\gamma_{pz_j}(z,\cdot)$ are positively one- and zero-homogeneous, respectively, 
for every $z \in \Omega$. 
In addition, we assume that $p \mapsto \gamma(z, p)$ is strictly convex for every $z \in \Omega$ in the sense that
\begin{equation} \label{eq:sconv}
\gamma_{pp}(z,p) q \cdot q >0 \qquad \forall\ z \in \Omega, \ p,q  \in \bR^2 \mbox{ with } |p|=|q| =1, p \cdot q =0.
\end{equation}

We are now in a position to define anisotropic curve shortening flow. 
To this end, with the help of Corollary~4.3 in \cite{DoganN12}, we first 
state the first variation of the functional $\mathcal E$ in (\ref{eq:aniso}),
the proof of which will be given in Appendix~\ref{sec:appA}.

\begin{lemma} \label{lem:firstvar}
Let $\Gamma \subset \Omega$ be a smooth curve with unit normal $\nu$,
unit tangent $\tau$ and scalar curvature $\varkappa$. 
Let $V$ be a smooth vector field
defined in an open neighbourhood of $\Gamma$. Then 
the first variation of $\mathcal E$ at $\Gamma$ in the direction $V$ is
given by
\begin{equation} \label{eq:firstvar2}
{\rm d} \mathcal E(\Gamma; V) = 
- \int_\Gamma \varkappa_\gamma \, V \cdot \nu_\gamma\, 
\altm \gamma(\cdot,\nu) \dH1,
\end{equation}
where  
\begin{equation} \label{eq:nugamma}
\nu_\gamma = \frac{\nu}{\gamma(\cdot,\nu)} \quad\text{and}\quad
 \varkappa_\gamma = \varkappa \gamma_{pp}(\cdot,\nu) \tau \cdot \tau 
- \gamma_{p_i z_i}(\cdot,\nu) 
- \frac{\nabla \altm}{\altm} \cdot \gamma_p(\cdot,\nu)
\quad \text{on } \Gamma
\end{equation}
denote the anisotropic normal and the anisotropic curvature of $\Gamma$,
respectively. 
\end{lemma}

We remark that the definitions in \eqref{eq:nugamma} correspond to 
(3.5) and (4.1) in \cite{BellettiniP96}. Note also that
$\nu_\gamma$ is a vector that is normal to $\Gamma$, but normalized in such a
way that $\gamma(z, \nu_\gamma(z)) = 1$, $z \in \Gamma$.
We remark that although $\varkappa_\gamma$ clearly depends on both $\gamma$ and 
$\altm$, we prefer to use the simpler notation that drops the dependence on
$\altm$.

Following \cite[(1.1)]{BellettiniP96}, we now consider a natural gradient flow
induced by \eqref{eq:firstvar2}.
In particular, given a family of curves $(\Gamma(t))_{t \in [0,T]}$ in
$\Omega$, we say that $\Gamma(t)$ evolves 
according to anisotropic curve shortening flow, provided that
\begin{equation}  \label{eq:acsf}
\mathcal{V}_\gamma = \varkappa_\gamma \quad \text{ on } \Gamma(t),
\end{equation}
where $\mathcal{V}_\gamma= (\mathcal{V} \nu) \cdot \nu_\gamma = 
\frac{1}{\gamma(\cdot,\nu)} \mathcal{V}$, 
with $\mathcal{V}$ denoting the normal velocity of $\Gamma(t)$,
and where $\varkappa_\gamma$ is defined in \eqref{eq:nugamma}.
We remark that the name of the flow is inspired by the fact that 
solutions of \eqref{eq:acsf} satisfy the energy relation
\begin{equation} \label{eq:anigradflow}
\ddt \int_{\Gamma(t)} a \, \gamma(\cdot,\nu) \dH1 + \int_{\Gamma(t)}
 | \mathcal V_\gamma |^2 \, a \,  \gamma(\cdot,\nu) \dH1  = 0.
\end{equation}
We note that the higher dimensional analogue of \eqref{eq:acsf} is usually
called anisotropic mean curvature flow or anisotropic motion by mean curvature.
Hence alternative names for the evolution law \eqref{eq:acsf} in the 
planar case treated in this paper are
anisotropic curvature flow and anisotropic motion by curvature.

\renewcommand{\labelenumi}{(\alph{enumi})}
\renewcommand{\theenumi}{(\alph{enumi})}
\begin{example} \label{ex:1}
\rule{0pt}{0pt}
\begin{enumerate}[noitemsep,topsep=-5pt]
\item \label{item:ex1a}
Isotropic case: 
We let $\gamma(z,p)=|p|$ and $\altm(z)= 1$ for all $z \in \Omega = \bR^2$, 
recall \eqref{eq:iso}, so that $\mathcal E(\Gamma)$ is the length of $\Gamma$. 
In this case \eqref{eq:acsf} is just the well-known curve shortening flow,
\begin{equation*} % \label{eq:csf}
\mathcal{V} = \varkappa\quad \text{ on } \Gamma(t).
\end{equation*}
\item \label{item:ex1b}
Space-independent anisotropy: 
We let $\gamma(z,p)= \gamma_0(p)$ and $\altm(z)= 1$
for all $z \in \Omega = \bR^2$, recall \eqref{eq:homogeneous}, so that
$\mathcal E(\Gamma) = \int_\Gamma \gamma_0(\nu) \dH1$ is the 
associated anisotropic length. Then \eqref{eq:acsf} reduces to
\begin{equation} \label{eq:acsf0}
\frac{1}{\gamma_0(\nu)} \mathcal V = \varkappa_{\gamma_0} = 
\varkappa \gamma_0''(\nu) \tau \cdot \tau
\quad \text{ on } \Gamma(t),
\end{equation}
where here and throughout $\gamma_0'$ and $\gamma_0''$ denote the gradient and
Hessian of $\gamma_0$, respectively.
We observe that \eqref{eq:acsf0} corresponds to \cite[(8.20)]{DeckelnickDE05} 
with $\beta(\nu)= \frac1{\gamma_0(\nu)}$, see also \cite{Pozzi12} for a nice 
derivation of this law. 
Of course, for $\gamma_0(p) = |p|$ we obtain the isotropic case
\ref{item:ex1a}.
\item \label{item:ex1c}
Riemannian manifolds: 
Suppose that $(\mathcal M,g)$ is a two-dimensional Riemannian manifold. 
Let $F: \Omega \to \mathcal M$ be a local parameterization of $\mathcal M$,
$\{ \partial_1,\partial_2 \}$ the corresponding basis of the tangent space
$T_{F(z)} \mathcal M$ and
$g_{ij}(z)=g_{F(z)}(\partial_i,\partial_j)$, $z \in \Omega$,
as well as $G(z)=(g_{ij}(z))_{i,j=1}^2$.
We set $\gamma(z,p)= \sqrt{G^{-1}(z) p \cdot p}$ and 
$\altm(z)=\sqrt{\det G(z)}$. Then we have
\begin{equation} \label{eq:mgamma}
\altm(z) \gamma(z,p) = \sqrt{ \det G(z) G^{-1}(z) p \cdot p} 
= \sqrt{ G(z) p^\perp \cdot p^\perp},
\end{equation}
where $p^\perp = \binom{p_1}{p_2}^\perp = \binom{-p_2}{p_1}$ denotes
an anti-clockwise rotation of $p$ by $\frac{\pi}{2}$.
For a curve $\Gamma \subset \Omega$ the vector $\tau = -\nu^\perp$ then is a 
unit tangent and 
\begin{equation} \label{eq:Egamma}
\mathcal E(\Gamma) = \int_\Gamma \altm\gamma(\cdot,\nu) \dH1
= \int_\Gamma \sqrt{G \tau \cdot \tau} \dH1
\end{equation}
is the Riemannian length of the curve 
$\tilde \Gamma = F(\Gamma) \subset \mathcal M$. 
We show in Appendix~\ref{sec:appB} that
the geodesic curvature of $\tilde\Gamma$ at $F(z)$ is equal to
$\varkappa_\gamma$ at $z \in \Gamma$, and also that 
$(\Gamma(t))_{t \in [0,T]} \subset \Omega$ is a solution of 
\eqref{eq:acsf}, if and only if %$(\tilde \Gamma(t))_{t \in [0,T]}$, 
$\tilde \Gamma(t)=F(\Gamma(t))$
evolves according to geodesic curvature flow in $\mathcal M$.  
\end{enumerate}
\end{example}

\setcounter{equation}{0}
\section{DeTurck's trick for anisotropic curve shortening flow} 
\label{sec:deturck}

In what follows we shall employ a parametric description of the evolving curves. 
Let $\Gamma(t)=x(I,t)$, where $x:I \times [0,T] \rightarrow \bR$ and 
$I=\RZ$. % is the periodic interval $[0,1]$.
In order to satisfy (\ref{eq:acsf}) we require that
\begin{equation} \label{eq:acsfIlift}
\frac{1}{\gamma(x,\nu \circ x)} x_t \cdot (\nu \circ x) = 
\varkappa_\gamma \circ x \qquad \text{in } I \times (0,T].
\end{equation}
{From} now on we fix a normal on $\Gamma(t)$ induced by the parameterization
$x$, and, as no confusion can arise, we identify $\nu \circ x$ with $\nu$,
$\varkappa_\gamma \circ x$ with $\varkappa_\gamma$ and similarly 
$\varkappa \circ x$ with $\varkappa$. In particular, 
we define the unit tangent, the unit normal and the curvature of $\Gamma(t)$ by
\begin{equation} \label{eq:tau}
\tau = 
\frac{x_\rho}{| x_\rho|}, \quad  \nu = \tau^\perp, \quad \varkappa =
\frac1{|  x_\rho|} \left( \frac{x_\rho}{| x_\rho|} \right)_\rho \cdot \nu 
= \frac{x_{\rho \rho}}{| x_\rho|^2} \cdot \nu.
\end{equation}
In place of \eqref{eq:acsfIlift} we simply write
\begin{equation} \label{eq:acsfI}
\frac{1}{\gamma(x,\nu)} x_t \cdot \nu = \varkappa_\gamma
\quad \text{in } I \times (0,T].
\end{equation}
Clearly, \eqref{eq:acsfI} only prescribes the normal component of the
velocity vector $x_t$, and so there is a certain freedom in the tangential 
direction. 
Our aim is to introduce a strictly parabolic system of partial differential
equations for the parameterization $x$, 
whose solution in normal direction still satisfies (\ref{eq:acsfI}).

By way of motivation, let us briefly review the DeTurck trick in the isotropic 
setting, recall \eqref{eq:iso} and Example~\ref{ex:1}\ref{item:ex1a}. 
Then \eqref{eq:acsfI} collapses to $x_t \cdot \nu = \varkappa$, and adjoining
a zero tangential velocity leads to the formulation 
$x_t = \varkappa \nu = \frac1{|x_\rho|}(\frac{x_\rho}{|x_\rho|})_\rho$,
as the isotropic equivalent to \eqref{eq:anisosystem}.
We recall that optimal error bounds for a semidiscrete continuous-in-time 
finite element approximation of this formulation have been obtained 
in the seminal paper \cite{Dziuk94} by Gerd Dziuk.
One difficulty of Dziuk's original approach is that
the analysed system is degenerate in the tangential direction.
DeTurck's trick addresses this problem by removing the degeneracy 
through a suitable reparameterization. 
In fact, it is natural to consider the system 
\begin{equation} \label{eq:csfdeTurck}
x_t = \frac{x_{\rho \rho}}{|x_\rho|^2} ,
\end{equation}
recall \eqref{eq:tau}, 
for which a semidiscretization by linear finite elements was analyzed in
\cite{DeckelnickD95}. The appeal of this approach is that the analysis is very
elegant and simple. For example, the weak formulation of \eqref{eq:csfdeTurck} 
is given by
\begin{equation} \label{eq:DD95}
\int_I | x_\rho|^2 x_t \cdot \eta\drho 
+ \int_I x_\rho \cdot \eta_\rho \drho =0 \qquad \forall\ \eta \in [H^1(I)]^2,
\end{equation}
and choosing $\eta=x_t$ immediately gives rise to the estimate
\begin{equation} \label{eq:isoDeTurck}
\tfrac12 \ddt \int_I |x_\rho |^2 \drho 
= \int_I x_\rho \cdot x_{t \rho} \drho
= - \int_I |x_\rho |^2 |x_t|^2 \drho \leq 0,
\end{equation}
which can be mimicked on the discrete level.

Our starting point for extending DeTurck's trick to the anisotropic setting
is to define the function $\Phi: \Omega \times \bR^2 \to \bRgeq$ by setting
\begin{equation} \label{eq:Phi}
\Phi(z,p) = \tfrac12 \altm^2(z) \gamma^2(z,p^\perp)
\qquad \forall\ z \in \Omega,\ p \in \bR^2.
\end{equation}
We mention that the square of the anisotropy function plays an important role 
in the phase field approach to anisotropic mean curvature flow, 
cf.\ \cite{ElliottS96,AlfaroGHMS10,eck}. 

On noting \eqref{eq:tau}, \eqref{eq:phi1} and \eqref{eq:Phi} we compute, similarly to \eqref{eq:isoDeTurck}, that
\begin{align} \label{eq:motivation}
\tfrac12 \ddt \int_I \altm^2(x) \gamma^2(x,\nu) | x_\rho |^2 \drho & = \ddt \int_I \Phi( x,  x_\rho) \drho 
= \int_I \Phi_p(x,x_\rho) \cdot x_{t \rho} + \Phi_z(x,x_\rho) \cdot x_t \drho \nonumber \\
& = - \int_I \bigl( [\Phi_p(x, x_\rho)]_\rho - \Phi_z(x,x_\rho) \bigr) \cdot x_t \drho.
\end{align}
The crucial idea is now to define positive definite matrices 
$H(z,p) \in \bR^{2 \times 2}$, for 
$(z,p) \in \Omega \times (\bR^2 \setminus \{ 0 \})$, such that
if a sufficiently smooth $x$ satisfies
\begin{equation} \label{eq:Hxt}
H(x,x_\rho) x_t = [\Phi_p(x, x_\rho)]_\rho - \Phi_z(x,x_\rho) 
\quad \text{in } I \times (0,T],
\end{equation}
then $x$ is a solution to anisotropic curve shortening flow, \eqref{eq:acsfI}. 
For the construction of the matrices $H$ it is important to relate the right
hand side in \eqref{eq:Hxt} to the right hand side in \eqref{eq:acsfI}.
We begin by calculating
\begin{align*} 
\Phi_p(z,p) &= -\altm^2(z) \gamma(z,p^\perp) \gamma_p^\perp(z,p^\perp) 
\qquad \forall\ z \in \Omega,\ p \in \bR^2 \setminus\{0\},\\ %\label{eq:defa}\\
\Phi_z(z,p) &= \altm^2(z) \gamma(z,p^\perp) \gamma_z(z,p^\perp) 
+ \altm(z) \gamma^2(z,p^\perp) \nabla \altm(z)
\qquad \forall\ z \in \Omega,\ p \in \bR^2, %\label{eq:defb} 
\end{align*}
where we use the notation $\gamma_p^\perp(z,p) = [\gamma_p(z,p)]^\perp$. 
Furthermore, we obtain with the help of \eqref{eq:phi1}, \eqref{eq:phidd} and \eqref{eq:tau} that
\begin{align*} % \label{eq:aniDT}
[\Phi_p(x,x_\rho)]_\rho &
= -[\altm(x) \gamma(x,x_\rho^\perp) \altm(x) \gamma_p^\perp(x,x_\rho^\perp)]_\rho \\
& = -[\altm(x) \gamma(x,x_\rho^\perp)]_\rho \altm(x) \gamma_p^\perp(x,x_\rho^\perp) - \altm^2(x)\gamma(x,x_\rho^\perp) x_{j,\rho} \gamma_{pz_j}^\perp(x,x_\rho^\perp) \\
& \quad - \altm(x) \gamma(x,x_\rho^\perp) \nabla \altm(x) \cdot x_\rho \gamma_p^\perp(x,x_\rho^\perp)  
- \altm^2(x) \gamma(x,x_\rho^\perp) 
( \gamma_{pp}(x,x_\rho^\perp) x_{\rho \rho}^\perp )^\perp \\
& = -[\altm(x) \gamma(x,x_\rho^\perp)]_\rho \altm(x) \gamma_p^\perp(x,\nu) 
+ \altm^2(x) \gamma(x,\nu) | x_\rho |^2 
\bigl( \nu_1 \gamma_{pz_2}^\perp(x,\nu) - \nu_2 \gamma_{pz_1}^\perp(x,\nu) \bigr) \\
& \quad - \altm(x) | x_\rho | \gamma(x,\nu) \nabla \altm(x) \cdot x_\rho \gamma_p^\perp(x,\nu)
+ \altm^2(x) | x_\rho |^2  \gamma(x,\nu) \varkappa (\gamma_{pp}(x,\nu) \tau \cdot \tau) \nu.
\end{align*}
Similarly,
\begin{displaymath}
\Phi_z(x,x_\rho)= \altm^2(x) | x_\rho |^2 \gamma(x,\nu) \gamma_z(x,\nu) 
+ \altm(x) | x_\rho |^2 \gamma^2(x,\nu) \nabla \altm(x),
\end{displaymath}
and therefore
\begin{align}\label{eq:amb} 
[\Phi_p(x,x_\rho)]_\rho - \Phi_z(x,x_\rho) &
= -[\altm(x) \gamma(x,x_\rho^\perp)]_\rho \altm(x) \gamma_p^\perp(x,\nu) 
+ \altm^2(x) | x_\rho |^2  \gamma(x,\nu) \varkappa (\gamma_{pp}(x,\nu) 
\tau \cdot \tau) \nu \nonumber \\ & \quad 
+ \altm^2(x) \gamma(x,\nu) | x_\rho |^2 
\bigl( \nu_1 \gamma_{pz_2}^\perp(x,\nu) - \nu_2 \gamma_{pz_1}^\perp(x,\nu) 
- \gamma_z(x,\nu) \bigr) \nonumber \\
& \quad - \altm(x) | x_\rho |^2 \gamma(x,\nu) 
\bigl( \nabla \altm(x) \cdot \tau \gamma_p^\perp(x,\nu) 
+ \gamma(x,\nu) \nabla \altm(x) \bigr).
\end{align}
Observing that $\gamma_{z_i}(x,\nu) = \gamma_{pz_i}(x,\nu) \cdot \nu$,
recall \eqref{eq:phidd}, we find
\begin{displaymath}
[ \nu_1 \gamma_{pz_2}^\perp - \nu_2 \gamma_{pz_1}^\perp - \gamma_z]_1
= -\nu_1 \gamma_{p_2 z_2} + \nu_2 \gamma_{p_2 z_1} - \nu_1 \gamma_{p_1 z_1} - \nu_2 \gamma_{p_2 z_1}
= -(\gamma_{p_1 z_1} + \gamma_{p_2 z_2} ) \nu_1,
\end{displaymath}
and a similar argument for the second component yields
\begin{equation} \label{eq:nugp}
\nu_1 \gamma_{pz_2}^\perp(x,\nu) - \nu_2 \gamma_{pz_1}^\perp(x,\nu) 
- \gamma_z(x,\nu) = - \gamma_{p_i z_i}(x,\nu) \nu.
\end{equation}
Next, since 
\begin{equation} \label{eq:gpdottau}
\gamma_p^\perp(x,\nu) \cdot \tau = -\gamma_p(x,\nu) \cdot \nu = -\gamma(x,\nu)
\quad \text{and} \quad
\gamma_p^\perp(x,\nu) \cdot \nu = \gamma_p(x,\nu) \cdot \tau,
\end{equation}
we derive
\begin{align} \label{eq:nablaa}
\nabla \altm  \cdot \tau \gamma_p^\perp + \gamma \nabla \altm 
& = (\nabla \altm \cdot \tau) \left( (\gamma_p^\perp \cdot \tau) \tau 
+ (\gamma_p^\perp \cdot \nu) \nu \right) 
+ \gamma \left( (\nabla \altm \cdot \tau) \tau 
+ (\nabla \altm \cdot \nu) \nu \right) \nonumber \\
& = ( \nabla \altm \cdot \tau) ( \gamma_p \cdot \tau) \nu 
+ (\nabla \altm \cdot \nu) (\gamma_p \cdot \nu) \nu 
= (\nabla \altm \cdot \gamma_p) \nu.
\end{align}
If we insert \eqref{eq:nugp} and \eqref{eq:nablaa} into (\ref{eq:amb}), 
recall \eqref{eq:nugamma} and abbreviate 
$\omega(x)= [ \altm(x) \gamma(x,x_\rho^\perp)]_\rho$, we obtain
\begin{equation} \label{eq:amb1}
[\Phi_p(x,x_\rho)]_\rho - \Phi_z(x,x_\rho) 
= -\omega(x)\altm(x) \gamma_p^\perp(x,\nu) 
+ \altm^2(x) | x_\rho |^2 \gamma(x,\nu) \varkappa_\gamma \nu.
 \end{equation}

Let us now assume that $x$ is a solution of \eqref{eq:Hxt}, where $H(z,p)$ is an invertible matrix of the form
\begin{equation*} % \label{eq:HKlaus}
H(z,p) = \begin{pmatrix} 
\alpha(z,p) & - \beta(z,p) \\ \beta(z,p) & \alpha(z,p) 
\end{pmatrix}
\qquad \forall\ z \in \Omega,\ p \in \bR^2\setminus\{0\}.
\end{equation*}
In order to determine $\alpha(z,p), \beta(z,p) \in \bR$ such that $x$
satisfies \eqref{eq:acsfI}, we calculate
\begin{align*}
x_t \cdot \nu & = H^{-1}(x,x_\rho) \Bigl( [\Phi_p(x,x_\rho)]_\rho - \Phi_z(x,x_\rho) \Bigr) \cdot \nu \\
& = \frac{1}{\alpha^2(x,x_\rho)+ \beta^2(x,x_\rho)} \Bigl( [\Phi_p(x,x_\rho)]_\rho - \Phi_z(x,x_\rho) \Bigr) \cdot H(x,x_\rho) \nu.
\end{align*}
If we multiply by $\alpha^2(x,x_\rho) + \beta^2(x, x_\rho)$ and insert 
\eqref{eq:amb1} we obtain, on noting 
$\binom{-\nu_2}{\nu_1} = \nu^\perp = -\tau$ and \eqref{eq:gpdottau}, that
\begin{align*} % \label{eq:alphabetanu}
& (\alpha^2(x,x_\rho) + \beta^2(x, x_\rho)) x_t \cdot \nu \nonumber \\ & \quad
= \bigl( -\omega(x)\altm(x) \gamma_p^\perp(x,\nu) 
+ \altm^2(x) | x_\rho |^2 \gamma(x,\nu) \varkappa_\gamma \nu \bigr) \cdot \bigl( \alpha(x,x_\rho) \nu - \beta(x,x_\rho) \tau \bigr) \nonumber \\ & \quad
= \omega(x) \altm(x)  \bigl( -\alpha(x,x_\rho) \gamma_p(x,\nu) \cdot \tau - \beta(x,x_\rho)  \gamma(x,\nu) \bigr)
+ \alpha(x,x_\rho) \altm^2(x) | x_\rho |^2 \gamma(x,\nu) \varkappa_\gamma \nonumber \\ & \quad
= \alpha(x,x_\rho) \altm^2(x) | x_\rho |^2 \gamma(x,\nu) \varkappa_\gamma,
\end{align*}
provided that 
$\alpha(x,x_\rho) \gamma_p(x,\nu) \cdot \tau + \beta(x,x_\rho)  \gamma(x,\nu)=0$.
With this choice we obtain that $\displaystyle \beta^2(x,x_\rho) = \alpha^2(x,x_\rho) 
\frac{(\gamma_p(x,\nu) \cdot \tau)^2}{\gamma^2(x,\nu)}$,
and so
\begin{align*}
\frac{1}{\gamma(x,\nu)} x_t \cdot \nu & = \frac{\alpha(x,x_\rho)}{\alpha^2(x,x_\rho) + \beta^2(x,x_\rho)} \altm^2(x) | x_\rho |^2  
\varkappa_\gamma \\
& = \frac{1}{\alpha(x,x_\rho)} \frac{\gamma^2(x,\nu)}{\gamma^2(x,\nu) + (\gamma_p(x,\nu) \cdot \tau)^2} | x_\rho |^2  \altm^2(x) \varkappa_\gamma 
& = \frac{1}{\alpha(x,x_\rho)} 
\frac{\gamma^2(x, |x_\rho| \nu)}{| \gamma_p(x,\nu)|^2} \altm^2(x) \varkappa_\gamma,
\end{align*}
where in the last step we have used the one-homogeneity of $\gamma$, recall
\eqref{eq:phidd}. 
Clearly, (\ref{eq:acsfI}) will now hold if we choose
\[
\alpha(z,p)  = 
\frac{\altm^2(z) \gamma^2(z,p^\perp)}{| \gamma_p(z,p^\perp) |^2}, \quad
\beta(z,p)  = - \alpha(z,p) 
\frac{\gamma_p(z,p^\perp) \cdot p}{\gamma(z,p^\perp)}= -\frac{\altm^2(z) \gamma(z,p^\perp) \gamma_p(z,p^\perp) \cdot p}{| \gamma_p(z,p^\perp) |^2}.
\]
In summary, we have shown the following result.

\begin{lemma} % \label{lem:Hxt}
Let $\Phi(z,p) = \tfrac12 \altm^2(z) \gamma^2(z,p^\perp)$ and 
\begin{equation} \label{eq:Hfinal}
H(z,p)
= \frac{\altm^2(z) \gamma(z,p^\perp)}{| \gamma_p(z,p^\perp) |^2}
\begin{pmatrix} 
\gamma(z,p^\perp) & \gamma_p(z,p^\perp) \cdot p \\
- \gamma_p(z,p^\perp) \cdot p & \gamma(z, p^\perp) 
\end{pmatrix} \qquad \forall\ z \in \Omega,\ p \in \bR^2\setminus\{0\}.
\end{equation}
If $x : I \times [0,T] \to \Omega$ satisfies \eqref{eq:Hxt}, 
then $x$ is a solution to anisotropic curve shortening flow, \eqref{eq:acsfI}. 
In addition $H$ is positive definite in $\Omega \times (\bR^2\setminus\{0\})$ 
with
\begin{equation} \label{eq:Hposdef}
H(z,p) w \cdot w = \frac{\altm^2(z) \gamma^2(z,p^\perp)}
{| \gamma_p(z,p^\perp) |^2} |w|^2 \qquad \forall\
z \in \Omega,\ p \in \bR^2 \setminus \{ 0 \},\ w \in \bR^2.
\end{equation}
\end{lemma}

Furthermore, it can be shown that the system \eqref{eq:Hxt} is strictly
parabolic. The proof hinges on the fact that $H$ and $\Phi_{pp}$ are positive
definite matrices in $\Omega \times (\bR^2 \setminus \{ 0 \})$. This property
of $\Phi_{pp}$ immediately follows from our convexity assumptions on $\gamma$, 
recall \eqref{eq:sconv}. 

\begin{lemma} 
Let $K \subset \Omega \times (\bR^2 \setminus \{ 0 \})$ be compact. 
Then there exists $\sigma_K>0$ such that 
\begin{equation} \label{eq:ellipt}
\Phi_{pp}(z,p) w \cdot  w \geq \sigma_K | w |^2 
\qquad \forall\ (z,p) \in K, \, w \in \bR^2.
\end{equation}
Furthermore, 
\begin{equation} \label{eq:lambdaL}
\Phi(z,q) - \Phi(z,p) - \Phi_p(z,p) \cdot (q -p) \geq 
\tfrac12 \sigma_K |  q -  p |^2 
\qquad \forall\ (z,p), (z,q) \in K \mbox{ with } \{ z \} \times [p,q] \subset K.
\end{equation}
Here, $[p,q] \subset \bR^2$ denotes the line segment connecting $p$ and $q$.
\end{lemma}
\begin{proof} 
It is shown in \cite[Remark~1.7.5]{Giga06} that \eqref{eq:sconv} implies that
$\Phi_{pp}(z,p)$ is positive definite for all $z \in \Omega$ and $p \neq 0$.
The bound (\ref{eq:ellipt}) then follows with the help of a compactness
argument, while the elementary identity
\begin{align*} % \label{eq:MVT}
\Phi(z,q) - \Phi(z,p) - \Phi_p(z,p) \cdot (q -p) & = \int_0^1 \bigl( 
\Phi_p(z, s q + (1-s) p) - \Phi_p(z,p) \bigr) 
\cdot (q -  p) \ds \nonumber \\ \qquad
& = \int_0^1 \int_0^1 s 
\Phi_{pp}(z, \theta s q + (1 - \theta s)  p) (q -p) 
\cdot ( q - p) \dtheta \ds ,
\end{align*}
together with \eqref{eq:ellipt}, implies (\ref{eq:lambdaL}). 
\end{proof}

\begin{lemma} \label{lem:petrovsky}
The system \eqref{eq:Hxt} is parabolic in the sense of Petrovsky.
\end{lemma}
\vspace{-4mm}
\begin{proof}
On inverting the matrix $H(x,x_\rho)$ we may write (\ref{eq:Hxt}) in the form
\begin{equation*} %\label{eq:Hxt1}
 x_t = H^{-1}(x,x_\rho) \Phi_{pp}(x,x_\rho) x_{\rho \rho} + H^{-1}(x,x_\rho) \bigl( \Phi_{pz}(x,x_\rho) x_\rho -
\Phi_z(x,x_\rho) \bigr) \quad \text{in } I \times (0,T].
\end{equation*}
Hence, by definition we need to show that the eigenvalues
of $H^{-1}(z,p) \Phi_{pp}(z,p)$ have positive real parts for every 
$(z,p) \in \Omega \times (\bR^2 \setminus \{ 0 \})$, 
see e.g.\ \cite[Definition~1.2]{EidelmanIK04}. Let us fix
$(z,p)$ and abbreviate $H=H(z,p)$, $A= \Phi_{pp}(z,p)$. 
The two eigenvalues $\lambda_1,\lambda_2 \in \bC$ of 
$H^{-1} A \in \bR^{2\times2}$ satisfy 
\[
\lambda_1 \lambda_2  = \det(H^{-1} A) = \frac{\det A}{\det H}>0,\quad
\lambda_1 + \lambda_2  = \tr (H^{-1} A) = \frac{\tr(H^T A)}{\det H} = \frac{H_{11} \tr A}{\det H} >0,
\]
since $H_{11}>0$ and $\det H >0$, recall \eqref{eq:Hfinal}, and since
$A$ is symmetric positive definite in view of \eqref{eq:ellipt}.
Hence, either both eigenvalues are positive real numbers, or 
$\lambda_2 = \overline{\lambda_1}$ with
$2 \text{Re} \lambda_1 =\lambda_1+\lambda_2 > 0$.
\end{proof}

In view of Lemma~\ref{lem:petrovsky} we expect that it is possible to prove the
short-time existence of a unique smooth solution to \eqref{eq:Hxt}.
Moreover, 
existence and uniqueness of classical smooth solutions to PDEs of the form
$x_t = \mathfrak b(\varkappa, \nu) \nu + \mathfrak a \tau$, arising
from closely related curvature driven geometric evolution equations,
have been obtained in \cite{MikulaS01}. 

The weak formulation of \eqref{eq:Hxt} now reads as follows.
Given $x_0 : I \to \Omega$, find $x:I \times [0,T] \rightarrow \Omega$ such 
that $x(\cdot,0)= x_0$ and, for $t \in (0,T]$,
\begin{equation} \label{eq:Hxtweak}
\int_I H(x,x_\rho) x_t \cdot \eta \drho
+ \int_I \Phi_p(x,x_\rho) \cdot \eta_\rho \drho + \int_I \Phi_z(x,x_\rho) \cdot \eta \drho = 0 \qquad \forall\ \eta \in [H^1(I)]^2.
\end{equation}

Choosing $\eta = x_t$ in \eqref{eq:Hxtweak} yields, on recalling
\eqref{eq:motivation} and \eqref{eq:Hposdef}, that
\begin{align} \label{eq:stab}
\tfrac12 \ddt \int_I \altm^2(x) \gamma^2(x,\nu) | x_\rho |^2 \drho & 
= \ddt \int_I \Phi( x,  x_\rho) \drho 
= \int_I \bigl(\Phi_p(x,x_\rho) \cdot x_{t\rho} + \Phi_z(x,x_\rho) \cdot x_t 
\bigr) \drho \nonumber \\ & 
= - \int_I  H(x,x_\rho) x_t \cdot x_t \drho  \leq 0.
\end{align}
Clearly \eqref{eq:stab} is the desired anisotropic analogue to
\eqref{eq:isoDeTurck}. So together with the fact that $x$ is a solution
of the gradient flow \eqref{eq:acsfI}, recall also \eqref{eq:anigradflow},
we obtain that both 
$\frac{1}{2} \int_I \altm^2(x) \gamma^2(x,\nu) | x_\rho |^2 \drho$ and
$\int_I a(x) \gamma(x,\nu) | x_\rho| \drho$ are
monotonically decreasing in time.

\begin{example} \label{ex:2}
We refer to the same numbering as in Example~\ref{ex:1}.
\begin{enumerate}[itemsep=1pt,topsep=-5pt]
\item % \label{item:ex2a}
Isotropic case: 
We have $\Phi(z,p)= \frac{1}{2} |p|^2$ and $H(z,p) = | p |^2 \Id$, so that 
\eqref{eq:Hxtweak} collapses to \eqref{eq:DD95}, which is the same as
(12) in \cite{DeckelnickD95}.
\item \label{item:ex2b}
Space-independent anisotropy: 
We have $\Phi(z,p)=\Phi_0(p)=\frac{1}{2} \gamma_0^2(p^\perp)$, so that 
\begin{subequations} %\label{eq:Phi0H0}
\begin{align} % \label{eq:Phi0} 
\Phi_p(z,p) = \Phi_0'(p) & = - \gamma_0(p^\perp) [\gamma_0'(p^\perp)]^\perp\\
\text{and } 
H(z,p)= H_0(p) & =
\frac{\gamma_0(p^\perp)}{| \gamma_0'(p^\perp) |^2} 
\begin{pmatrix} \gamma_0(p^\perp) & \gamma_0'(p^\perp) \cdot p \\
 -\gamma_0'(p^\perp) \cdot p & \gamma_0(p^\perp) \end{pmatrix}.
\label{eq:H0}
\end{align}
\end{subequations}
Hence the weak formulation reads 
\begin{equation} \label{eq:ex2b}
\int_I H_0(x_\rho) x_t \cdot \eta \drho 
+ \int_I \Phi_0'(x_\rho) \cdot \eta_\rho \drho =0 
\qquad \forall\ \eta \in [H^1(I)]^2.
\end{equation}
\item \label{item:ex2c}
Riemannian manifolds: 
In view of  \eqref{eq:mgamma} we have $\Phi(z,p)= \frac{1}{2} G(z) p \cdot p$, while
\begin{align} 
H(z,p) & = \frac{\det G(z) \bigl( G^{-1}(z) p^\perp \cdot p^\perp \bigr)^{\frac{3}{2}}}{| G^{-1}(z) p^\perp|^2}
\begin{pmatrix}
\sqrt{G^{-1}(z) p^\perp \cdot p^\perp} & \frac{G^{-1}(z) p^\perp \cdot p}{\sqrt{G^{-1}(z) p^\perp \cdot p^\perp}} \\
-\frac{G^{-1}(z) p^\perp \cdot p}{\sqrt{G^{-1}(z) p^\perp \cdot p^\perp}} & \sqrt{G^{-1}(z) p^\perp \cdot p^\perp} \end{pmatrix} \nonumber  \\
& = \frac{(\det G(z)) G(z) p \cdot p}{| G(z) p |^2} \begin{pmatrix} G(z) p \cdot p & -G(z) p \cdot p^\perp \\
G(z) p \cdot p^\perp & G(z) p \cdot p \end{pmatrix}. \label{eq:Hriem}
\end{align}
Hence the weak formulation reads
\begin{displaymath}
\int_I H(x,x_\rho) x_t \cdot \eta \drho + \int_I G(x) x_\rho \cdot \eta_\rho \drho + \tfrac{1}{2} \int_I
\eta_i G_{z_i}(x) x_\rho \cdot x_\rho  \drho=0
\qquad\forall\ \eta \in [H^1(I)]^2.
\end{displaymath}
\end{enumerate}
\end{example}

\begin{remark} % \label{rem:beta1}
It is a straightforward matter to extend our approach for \eqref{eq:acsfI} 
to the more general flow
\begin{equation} \label{eq:betahat}
\altbeta(x,\nu) x_t \cdot \nu = \varkappa_\gamma
\quad \text{in } I \times (0,T],
\end{equation}
compare also with \eqref{eq:betaflow} in the space-independent case. 
In particular, it can be easily shown that if $x$ is a solution to 
\begin{equation*} % \label{eq:tildeHxt}
\gamma(x,\nu) \altbeta(x,\nu)
H(x,x_\rho) x_t = [\Phi_p(x, x_\rho)]_\rho - \Phi_z(x,x_\rho) 
\quad \text{in } I\times (0,T],
\end{equation*}
then it automatically solves \eqref{eq:betahat}. 
Extending our analysis in Section~\ref{sec:fem} 
to this more general case is straightforward, 
on making the necessary smoothness assumptions on $\altbeta$.
\end{remark}

\setcounter{equation}{0}
\section{Finite element approximation} \label{sec:fem}

In order to define our finite element approximation, 
let $0=q_0 < q_1 < \ldots < q_{J-1}< q_J=1$  be a
decomposition of $[0,1]$ into intervals $I_j=[q_{j-1},q_j]$. Let $h_j=q_j - q_{j-1}$ as well as $h=\max_{1 \leq j \leq J} h_j$. We
assume that there exists a positive constant $c$ such that
\begin{equation*} %\label{eq:quasiuniform}
h \leq c h_j, \quad 1 \leq j \leq J,
\end{equation*}
so that the resulting family of partitions of $[0,1]$ is quasiuniform. 
Within $I$ we identify $q_J=1$ with $q_0=0$ and define the finite element 
spaces
\begin{displaymath}
V^h = \{\chi \in C^0(I) : \chi\!\mid_{I_j} \mbox{ is affine},\ j=1,\ldots, J\}
\quad\text{and}\quad \Vh = [V^h]^2.
\end{displaymath}
Let $\{\chi_j\}_{j=1}^J$ denote the standard basis of $V^h$.
For later use, we let $\pi^h:C^0(I)\to V^h$ 
be the standard interpolation operator at the nodes $\{q_j\}_{j=1}^J$,
and we use the same notation for the interpolation of vector-valued
functions. 
It is well-known that for 
$k \in \{ 0,1 \}$, $\ell \in \{ 1,2 \}$ and $p \in [2,\infty]$ it holds that
\begin{subequations}
\begin{alignat}{2}
h^{\frac 1p - \frac 1r} \| \eta_h \|_{0,r} 
+ h | \eta_h |_{1,p} & \leq C \| \eta_h \|_{0,p} 
\qquad && \forall\ \eta_h \in V^h, \qquad r \in [p,\infty], 
\label{eq:inverse} \\
| \eta - \pi^h \eta |_{k,p} & \leq Ch^{\ell-k} | \eta |_{\ell,p} 
\qquad && \forall\ \eta \in W^{\ell,p}(I). \label{eq:estpih} 
\end{alignat}
\end{subequations}

Our semidiscrete approximation of \eqref{eq:Hxtweak} is now given as follows.
Find $x_h:I \times [0,T] \rightarrow \Omega$ such that 
$x_h(\cdot,0)= \pi^h x_0$ and, for $t \in (0,T]$, $x_h(\cdot,t) \in \Vh$ such
that
\begin{equation} \label{eq:Hxtdiscrete}
\int_I H(x_h,x_{h,\rho}) x_{h,t} \cdot \eta_h \drho
+ \int_I \Phi_p(x_h,x_{h,\rho}) \cdot \eta_{h,\rho} \drho + \int_I \Phi_z(x_h,x_{h,\rho}) \cdot \eta_h \drho = 0 
\qquad \forall\ \eta_h \in \Vh.
\end{equation}
Expanding $x_h(\cdot,t)=\sum_{j=1}^J x_h(q_j,t) \chi_j$ 
we find that (\ref{eq:Hxtdiscrete}) gives rise to a system of 
ordinary differential equations (ODEs) in 
$\bR^{2J}$, which has a unique solution on some interval $[0,T_h)$. By choosing $\eta_h = x_{h,t}$ one also immediately obtains a semidiscrete analogue of (\ref{eq:stab}).

In what follows we assume that (\ref{eq:Hxt}) has a smooth solution
$x: I \times [0,T] \rightarrow \Omega$ satisfying 
\begin{equation} \label{eq:regul}
0 < c_0 \leq | x_\rho | \leq C_0 \quad \mbox{ in } I \times [0,T] \quad \mbox{ and } \int_0^T \| x_t \|_{0,\infty} \dt \leq C_0. 
\end{equation}
Let $S= x(I \times [0,T])$. Then there exists $\delta>0$ such that $\overline{B_\delta(S)} \subset \Omega$ and we define the compact set 
$K= \overline{B_\delta(S)} \times (\overline{B_{2C_0}(0)} \setminus B_{\frac{c_0}{2}}(0)) \subset \Omega \times (\bR^2 \setminus \{ 0 \})$.
We may choose $M_K \geq 0$ and $c_1>0$ such that 
\begin{align} 
& \max_{| \beta | \leq 3} \max_{(z,p) \in K} | D^\beta \gamma(z,p) | \leq M_K, 
\quad \max_{| \beta | \leq 2} \max_{z \in \overline{B_\delta(S)}}  | D^\beta \altm (z) | \leq M_K, \label{eq:upb}  \\
& \gamma(z,p) \geq c_1, \; | \gamma_p(z,p) | \geq c_1, \; \altm(z) \geq c_1 
\qquad \forall\ z \in \overline{B_\delta(S)},\ 
p \in \overline{B_{2C_0}(0)} \setminus B_{\frac{c_0}{2}}(0) . \label{eq:lob}
\end{align}

\begin{theorem} \label{thm:errest}
Suppose that \eqref{eq:Hxt} has a smooth solution 
$x:I \times [0,T] \rightarrow \Omega$ satisfying \eqref{eq:regul}.  
Then there exists $h_0>0$ such that for $0<h \leq h_0$ the semidiscrete problem
\eqref{eq:Hxtdiscrete} has a unique solution 
$x_h:I \times [0,T] \rightarrow \Omega$, and the following error bounds hold:
\begin{equation} \label{eq:errest}
\max_{t \in [0,T]} \| x(\cdot,t) - x_h(\cdot,t) \|_1^2 + \int_0^T \|  x_t - x_{h,t} \|_0^2 \dt
\leq C h^2.
\end{equation}
\end{theorem}
\begin{proof} Let us define
\begin{align*}
\hat T_h = \sup \Bigl\{ t \in [0,T] & :
x_h \text{ solves } (\ref{eq:Hxtdiscrete}) \text{ on } [0,t],
\text{ with }
\int_0^t \|  x_{h,t} \|_{0,\infty} \ds \leq 2 C_0\text{ and } \\ & \quad
 \| (x- x_h )(\cdot,s) \|_{0,\infty} \leq \delta,\ 
\| (  x_{\rho} -  x_{h,\rho})(\cdot,s) \|_{0,\infty} \leq \tfrac12 c_0,
\ 0 \leq s \leq t \Bigr\}.
\end{align*}
Let $(\rho,t) \in I \times [0,\hat T_h)$. 
Since $|(1-\lambda) x(\rho,t) + \lambda x_h(\rho,t) - x(\rho,t) | 
\leq \| (x-x_h)(\cdot,t) \|_{0,\infty} \leq \delta$ for all
$\lambda \in [0,1]$, we find that $[x(\rho,t),x_h(\rho,t)] \subset \overline{B_\delta(S)}$. Arguing 
in a similar way for the first derivative, we deduce that 
\begin{equation} \label{eq:subK}
[x(\rho,t),x_h(\rho,t)] \times [x_\rho(\rho,t),x_{h,\rho}(\rho,t)] \subset K 
\qquad \forall\ (\rho,t) \in I \times [0,\hat T_h).
\end{equation}
Comparing \eqref{eq:Hxtweak} and \eqref{eq:Hxtdiscrete} we see that
the error $e=  x - x_h$ satisfies
\begin{align}\label{eq:err0}
&
\int_I H(x_h,x_{h,\rho}) e_t \cdot \eta_h \drho + \int_I \bigl( \Phi_p(x_h, x_\rho) - \Phi_p(x_h, x_{h,\rho}) \bigr) \cdot \eta_{h,\rho} \drho \nonumber \\ & \quad
= \int_I \bigl( H(x_h,x_{h,\rho}) - H(x,x_\rho) \bigr)  x_t \cdot \eta_h \drho + \int_I \bigl( \Phi_p(x_h, x_\rho) -
\Phi_p(x,x_\rho) \bigr) \cdot \eta_{h,\rho} \drho 
\nonumber \\ & \qquad 
+ \int_I \bigl( \Phi_z(x_h, x_{h,\rho}) - \Phi_z(x,x_\rho) \bigr) \cdot \eta_h \drho 
\qquad \forall\ \eta_h \in \Vh.
\end{align}
Using the identity
\begin{equation} \label{eq:ide}
e = x - \pi^h  x + \pi^h  e
\end{equation}
and choosing  $\eta_h = \pi^h e_t$ in \eqref{eq:err0}, we obtain
\begin{align} \label{eq:err1}
& \int_I H(x_h,  x_{h,\rho}) e_t \cdot e_t \drho + \int_I \bigl( \Phi_p(x_h,x_\rho) - \Phi_p(x_h,x_{h,\rho}) \bigr) \cdot  e_{t\rho} \drho 
\nonumber \\ & \quad
= \int_I H(x_h,x_{h,\rho})  e_t \cdot(  x_t - \pi^h x_t) \drho + \int_I \bigl( \Phi_p(x_h,x_\rho) - \Phi_p(x_h,x_{h,\rho}) \bigr) \cdot 
(  x_t - \pi^h  x_t)_\rho \drho 
\nonumber \\ & \qquad
+ \int_I \bigl( H(x_h, x_{h,\rho}) - H( x, x_\rho) \bigr)  x_t \cdot  \pi^h e_t \drho + \int_I \bigl( \Phi_p(x_h,x_\rho) -
\Phi_p(x, x_\rho) \bigr) \cdot (\pi^h e_t)_\rho \drho 
\nonumber \\ & \qquad
+ \int_I \bigl( \Phi_z(x_h, x_{h,\rho}) - \Phi_z(x,x_\rho) \bigr) \cdot \pi^h e_t \drho =: \sum_{i=1}^5 S_i.
\end{align}
Let us begin with the two terms on the left hand side of (\ref{eq:err1}). 
Clearly, \eqref{eq:Hposdef},  \eqref{eq:upb}, \eqref{eq:lob} and \eqref{eq:subK} imply that
\begin{equation} \label{eq:t1}
\int_I H(x_h, x_{h,\rho}) e_t \cdot  e_t \drho = \int_I \frac{\altm^2(x_h) \gamma^2(x_h,x_{h,\rho}^\perp)}
{| \gamma_p(x_h,x_{h,\rho}^\perp) |^2} | e_t |^2 \drho \geq \widetilde c_0 \| e_t \|_0^2,
\end{equation}
where $\widetilde c_0=M_K^{-2} c_1^4$. Next, we write
\begin{align*} %\label{eq:err2}
& 
\bigl( \Phi_p(x_h,x_\rho) - \Phi_p(x_h, x_{h,\rho}) \bigr) \cdot  e_{t\rho} 
\nonumber \\ & \quad 
= \Phi_p(x_h, x_\rho) \cdot  x_{t\rho} + \Phi_p(x_h, x_{h,\rho}) \cdot  x_{h,t \rho} - \Phi_p(x_h, x_\rho) \cdot x_{h,t \rho}
- \Phi_p( x_h, x_{h,\rho}) \cdot x_{t\rho} \nonumber \\ & \quad
= \left[ \Phi(x_h, x_{h,\rho}) - \Phi_p(x_h, x_\rho) \cdot x_{h,\rho} \right]_t + \Phi_p(x_h, x_\rho) \cdot  x_{t\rho} 
- \Phi_p( x_h, x_{h,\rho}) \cdot  x_{t\rho}  
\nonumber  \\ & \qquad 
- \Phi_z(x_h, x_{h,\rho}) \cdot x_{h,t}  + \left[ \Phi_p(x_h, x_\rho) \right]_t \cdot  x_{h,\rho}.
\end{align*}
Since $p \mapsto \Phi(z,p)$ and $p \mapsto \Phi_{z_j}(z,p)$ are positively homogeneous of degree 2, we have
\begin{displaymath}
\Phi_p(x_h,x_\rho) \cdot x_\rho = 2 \Phi(x_h,x_\rho) \quad \mbox{ and } \quad
\Phi_{p z_j}(x_h,  x_\rho) \cdot x_\rho = 2 \Phi_{z_j}(x_h,  x_\rho),
\end{displaymath}
and therefore
\begin{align} \label{eq:err3}
&
\bigl( \Phi_p(x_h, x_\rho) - \Phi_p(x_h,x_{h,\rho}) \bigr) \cdot  e_{t\rho} \nonumber \\ & \quad
= \left[ \Phi(x_h, x_{h,\rho}) - \Phi(x_h, x_\rho) - \Phi_p(x_h,x_\rho) \cdot (x_{h,\rho} -  x_\rho) \right]_t - \left[ \Phi(x_h,x_\rho)
\right]_t \nonumber \\ & \qquad
+ \Phi_p(x_h, x_\rho) \cdot x_{t\rho} 
- \Phi_p( x_h, x_{h,\rho}) \cdot x_{t\rho}  - \Phi_z(x_h, x_{h,\rho}) \cdot x_{h,t}  + \left[ \Phi_p(x_h,x_\rho) \right]_t \cdot  x_{h,\rho} \nonumber \\ & \quad 
= \left[ \Phi(x_h, x_{h,\rho}) - \Phi(x_h, x_\rho) - \Phi_p( x_h, x_\rho) \cdot ( x_{h,\rho} - x_\rho) \right]_t - \bigl(
\Phi_z(x_h, x_{h,\rho}) + \Phi_z( x_h, x_\rho) \bigr) \cdot x_{h,t} \nonumber \\ & \qquad
-\Phi_p( x_h, x_{h,\rho}) \cdot x_{t\rho} + x_{h,j,t} \Phi_{p z_j}(x_h, x_\rho) \cdot  x_{h, \rho} + 
\Phi_{pp}(x_h,x_\rho)  x_{t\rho} \cdot  x_{h,\rho} 
\nonumber \\ & \quad
= \left[ \Phi( x_h, x_{h,\rho}) - \Phi(x_h, x_\rho) - \Phi_p(x_h,x_\rho) \cdot (x_{h,\rho} -  x_\rho) \right]_t
-\bigl( \Phi_p(x_h,x_{h,\rho})  - \Phi_{pp}(x_h, x_\rho)  x_{h,\rho} \bigr) \cdot x_{t\rho} \nonumber \\ & \qquad
- \bigl( \Phi_{z_j}(x_h, x_{h,\rho}) - \Phi_{z_j}(x_h, x_\rho) - \Phi_{p z_j}(x_h,x_\rho) \cdot (x_{h,\rho}- x_\rho) \bigr) x_{h,j,t}
 \nonumber \\ & \quad
= \left[ \Phi(x_h,x_{h,\rho}) - \Phi(x_h,x_\rho) - \Phi_p(x_h,x_\rho) \cdot ( x_{h,\rho} -  x_\rho) \right]_t 
\nonumber \\ & \qquad
 - \bigl( \Phi_p( x_h, x_{h,\rho}) - \Phi_p(x_h, x_\rho) - \Phi_{pp}( x_h, x_\rho)(x_{h,\rho} -  x_\rho) \bigr) \cdot  x_{t\rho}
 \nonumber \\ & \qquad
 - \bigl( \Phi_{z_j}(x_h, x_{h,\rho}) - \Phi_{z_j}(x_h,x_\rho) - \Phi_{p z_j}(x_h,x_\rho) \cdot ( x_{h,\rho}-  x_\rho) \bigr) x_{h,j,t},
\end{align}
where the last equality follows from the relation 
$\Phi_p(z,p) = \Phi_{pp}(z,p) p$, recall \eqref{eq:phidd}. 
If we combine (\ref{eq:err3}) with (\ref{eq:t1}) and use a Taylor expansion
together with \eqref{eq:upb} and \eqref{eq:subK}, we obtain for 
the left hand side of (\ref{eq:err1}), that 
\begin{align} \label{eq:err4}
&
\int_I H(x_h, x_{h,\rho}) e_t \cdot e_t \drho + \int_I \bigl( \Phi_p(x_h, x_\rho) - \Phi_p(x_h, x_{h,\rho}) \bigr) \cdot  e_{t\rho} \drho \nonumber \\ & \quad
\geq \widetilde c_0 \|  e_t \|_0^2 +
\ddt \int_I \Phi(x_h, x_{h,\rho}) - \Phi(x_h, x_\rho) - \Phi_p(x_h,x_\rho) \cdot ( x_{h,\rho} -  x_\rho) \drho \nonumber \\ & \qquad
- C \big( \|  x_{t\rho} \|_{0,\infty} + \| x_{h,t} \|_{0,\infty} \bigr) \|  e_\rho \|_0^2. 
\end{align}
Let us next estimate the terms on the right hand side of (\ref{eq:err1}). 
To begin, we obtain from \eqref{eq:Hfinal}, \eqref{eq:upb}, \eqref{eq:lob} and  \eqref{eq:estpih} that
\begin{subequations} \label{eq:s12345}
\begin{equation} %\label{eq:s1}
S_1 
\leq C \|  e_t \|_0 \|    x_t - \pi^h  x_t \|_0 
\leq C h \|  x_{t\rho} \|_0 \|  e_t \|_0 
\leq \epsilon \|  e_t \|_0^2 + C_\epsilon h^2 \|  x_{t\rho} \|_0^2.
\end{equation}
The remaining terms involve differences between $\Phi_p$, $H$ and $\Phi_z$, which will be estimated with the help of
\eqref{eq:upb} and \eqref{eq:subK}. Using  \eqref{eq:estpih} we have
\begin{equation} %\label{eq:s2}
S_2 
\leq C \|  e_\rho \|_0 \| ( x_t - \pi^h  x_t)_\rho \|_0 
\leq C h \|  x_{t\rho \rho} \|_0 \|  e_\rho \|_0 
\leq \|  e_\rho \|_0^2 + C h^2 \|  x_{t\rho \rho} \|_0^2,
\end{equation}
as well as
\begin{equation} %\label{eq:s3}
S_3 
\leq C \| e \|_1 \|  x_t \|_{0,\infty} \| \pi^h  e_t \|_0 
\leq C \| e \|_1 \bigl( \|  e_t \|_0 + 
       \|  x_t - \pi^h  x_t \|_0 \bigr) 
\leq \epsilon \|  e_t \|_0^2 + C_\epsilon \| e \|_1^2 
+ C h^2 \|  x_{t\rho} \|_0^2,
\end{equation}
and similarly
\begin{equation} %\label{eq:s5}
S_5 
\leq C \| e \|_1 \| \pi^h  e_t \|_0
\leq \epsilon \|  e_t \|_0^2 + C_\epsilon \| e \|_1^2 
+ C h^2 \|  x_{t\rho} \|_0^2.
\end{equation}
Finally, noting once again the identity \eqref{eq:ide} and the estimate
\eqref{eq:estpih}, we have 
\begin{align} %\label{eq:s4}
S_4 & =  \ddt \int_I \bigl( \Phi_p(x_h,x_\rho)- \Phi_p(x, x_\rho) \bigr) \cdot (\pi^h  e)_\rho \drho -
\int_I \left[ \Phi_p(x_h, x_\rho) - \Phi_p(x, x_\rho) \right]_t \cdot (\pi^h e)_\rho \drho \nonumber \\ & 
= \ddt \int_I \bigl( \Phi_p( x_h, x_\rho)- \Phi_p( x, x_\rho) \bigr) \cdot (\pi^h  e)_\rho \drho \nonumber \\ & \quad 
+ \int_I \left[ \Phi_{pz}(x_h, x_\rho)  e_t + (\Phi_{pz}( x, x_\rho)-\Phi_{pz}(x_h,x_\rho))  x_t \right] \cdot (\pi^h  e)_\rho \drho \nonumber \\ & \quad
+ \int_I \left[ (\Phi_{pp}(x,x_\rho)-\Phi_{pp}(x_h, x_\rho))  x_{t\rho} \right] \cdot (\pi^h  e)_\rho \drho \nonumber \\ & 
\leq \ddt \int_I \bigl( \Phi_p(x_h, x_\rho)- \Phi_p(x, x_\rho) \bigr) \cdot (\pi^h  e)_\rho \drho \nonumber \\ & \quad 
+ C \bigl( \|  e_t \|_0 + \|  e \|_{0,\infty} \|  x_t \|_1 \bigr) \bigl( \| e_\rho \|_0 + \|  x_\rho - (\pi^h  x)_\rho \|_0 \bigr) 
\nonumber \\ & 
\leq \ddt \int_I \bigl( \Phi_p( x_h, x_\rho)- \Phi_p( x, x_\rho) \bigr) \cdot (\pi^h  e)_\rho \drho + \epsilon \|  e_t \|_0^2 +
C_\epsilon \|  e \|_1^2 + C_\epsilon h^2 \|  x_{\rho \rho} \|_0^2,
\end{align}
\end{subequations}
where in the last inequality we have also used the embedding result
\eqref{eq:sobolev}.
If we insert (\ref{eq:err4}) and \eqref{eq:s12345} into 
(\ref{eq:err1}), and choose $\epsilon$ sufficiently small, we obtain
\begin{equation} \label{eq:err5}
\tfrac12 \widetilde c_0 \|  e_t \|_0^2 + \mu'(t) \leq C \bigl( 1+ \|  x_{h,t} \|_{0,\infty} \bigr) \|  e \|_1^2 + C h^2,
\end{equation}
where
\begin{equation*} % \label{eq:defmu}
\mu(t)= \int_I \Phi(x_h,x_{h,\rho}) - \Phi(x_h, x_\rho) - \Phi_p(x_h, x_\rho) \cdot ( x_{h,\rho} -  x_\rho) -
\bigl( \Phi_p( x_h, x_\rho)- \Phi_p( x, x_\rho) \bigr) \cdot (\pi^h  e)_\rho \drho.
\end{equation*}
Clearly, we have from \eqref{eq:lambdaL},
on noting \eqref{eq:ide} and \eqref{eq:estpih}, that
\begin{align*} % \label{eq:muest}
\mu(t) & 
\geq \tfrac12 \sigma_K \| e_\rho \|_0^2 - C \| e \|_0 \| (\pi^h e)_\rho \|_0 
\geq \tfrac12 \sigma_K \|  e_\rho \|_0^2 - C \|  e \|_0 
\bigl( \|  e_\rho \|_0 + \|  x_\rho - (\pi^h  x)_\rho \|_0 \bigr) \nonumber \\
& \geq \tfrac14 \sigma_K \|  e_\rho \|_0^2 - C_1 \|  e \|_0^2 - Ch^2 ,
\end{align*}
and hence
\begin{equation} \label{eq:h1est}
\|  e \|_1^2 
= \|  e \|_0^2 + \|  e_\rho \|_0^2 
\leq \| e \|_0^2 + \frac4{\sigma_K}\left[\mu(t) + C_1 \|  e \|_0^2 +
Ch^2\right]
\leq \frac{4}{\sigma_K} \bigl( C_2 \|  e \|_0^2 + \mu(t) \bigr) + C h^2,
\end{equation}
where $C_2= C_1 + \frac{\sigma_K}{4}$. 
Integrating (\ref{eq:err5}) with respect to time, and observing (\ref{eq:h1est}) as well as $\mu(0) \leq C h^2$, we derive
\begin{align*} % \label{eq:err6}
&
\tfrac12 \widetilde c_0 \int_0^t \|  e_t \|_0^2 \ds 
+ \bigl( \mu(t) + C_2 \|  e(t) \|_0^2 \bigr) \nonumber \\ & \quad
\leq C \int_0^t (1 + \|  x_{h,t} \|_{0,\infty} ) \|  e \|_1^2 \ds
+ 2 C_2 \int_0^t \|  e \|_0 \|  e_t \|_0 \ds +  C h^2  
\nonumber \\ & \quad
\leq C \int_0^t (1 + \|  x_{h,t} \|_{0,\infty} ) \bigl( \mu(s) + C_2 \|  e \|_0^2  \bigr) \ds + \tfrac14 \widetilde c_0 \int_0^t \|  e_t \|_0^2 \ds
+ C \int_0^t \|  e \|_0^2 \ds + C h^2,
\end{align*}
and hence
\begin{displaymath}
\mu(t) + C_2 \| e(t) \|_0^2  \leq Ch^2 + C \int_0^t (1 + \|  x_{h,t} \|_{0,\infty} ) \bigl( \mu(s) + C_2 \|  e \|_0^2  \bigr) \ds, \quad
0 \leq t < \hat T_h.
\end{displaymath}
Since $\int_0^{\hat T_h} \|  x_{h,t} \|_{0,\infty} \ds \leq 2 C_0$ by
definition, 
we deduce with the help of Gronwall's inequality and (\ref{eq:h1est}) that
\begin{equation} \label{eq:GW}
\int_0^{\hat T_h} \|  e_t \|_0^2 \ds 
+ \sup_{0 \leq s \leq \hat T_h} \|  e(s) \|_1^2 \leq C h^2.
\end{equation}
In particular, we have on recalling \eqref{eq:sobolev} that
\begin{displaymath}
\| x(\cdot,\hat T_h)- x_h(\cdot,\hat T_h) \|_{0,\infty} =  \|  e(\cdot,\hat T_h) \|_{0,\infty} \leq C \|  e(\cdot,\hat T_h)
\|_1 \leq Ch \leq \tfrac12 \delta,
\end{displaymath}
provided that $0<h \leq h_0$. 
Next, we have from \eqref{eq:estpih}, \eqref{eq:inverse}, \eqref{eq:ide} 
and \eqref{eq:GW} that
\begin{align*}
&
\| ( x_{\rho} -  x_{h,\rho})(\cdot,\hat T_h) \|_{0,\infty} 
\leq \| ( x_\rho - (\pi^h x)_\rho )(\cdot,\hat T_h) \|_{0,\infty} 
+ \| ( (\pi^h x)_\rho - x_{h,\rho} )(\cdot,\hat T_h) \|_{0,\infty}
\\ & \qquad 
\leq Ch + Ch^{-\frac{1}{2}} \| (x_{h,\rho} - (\pi^h x)_\rho)(\cdot,\hat T_h) \|_0 
\leq C h^{\frac{1}{2}} + C h^{-\frac{1}{2}} \| e_\rho(\cdot,\hat T_h) \|_0
\leq C h^{\frac{1}{2}},
\end{align*}
and similarly by \eqref{eq:regul} that
\begin{displaymath}
\int_0^{\hat T_h} \| x_{h,t} \|_{0,\infty} \ds \leq \int_0^{\hat T_h} \bigl( \| x_t \|_{0,\infty} \ds +  \| e_t \|_{0,\infty} \bigr) \ds \leq C_0+ 
C h^{-\frac{1}{2}} \int_0^{\hat T_h} \|  e_t \|_0 \ds \leq C_0 + C h^{\frac{1}{2}}.
\end{displaymath}
By choosing $h_0$ smaller if necessary we may therefore assume that 
$\| (x_{\rho} - x_{h,\rho})(\cdot,\hat T_h) \|_{0,\infty} 
\leq \frac{1}{4} c_0$
and $\int_0^{\hat T_h} \| x_{h,t} \|_{0,\infty} \ds \leq \frac{3}{2} C_0$. 
Suppose that $\hat T_h <T$. 
Then there exists an $\epsilon>0$ such that $x_h$ exists on
$[0,\hat T_h + \epsilon]$ with
$\| (x-x_h)(\cdot,t) \|_{0,\infty} \leq \delta$, 
$\| (x_{\rho} - x_{h,\rho})(\cdot,t) \|_{0,\infty} \leq \frac12 c_0$, 
for $0 \leq t \leq \hat T_h+\epsilon$,
and $\int_0^{\hat T_h + \epsilon} \|  x_{h,t} \|_{0,\infty} \ds \leq 2 C_0$, 
contradicting the definition of $\hat T_h$. 
Thus $\hat T_h =T$ and the theorem is proved. 
\end{proof}

\setcounter{equation}{0}
\section{Fully discrete schemes} \label{sec:fds}

{From} now on, let the $L^2$--inner product on $I$ be denoted by 
$(\cdot,\cdot)$. Due to the nonlinearities present in \eqref{eq:Hxtdiscrete},
for a fully practical scheme we need to introduce numerical quadrature. For our
purposes it is sufficient to consider classical mass lumping.
Hence for two piecewise continuous functions, with possible jumps at the 
nodes $\{q_j\}_{j=1}^J$, we define the mass lumped $L^2$--inner product 
$(u,v)^h$ via
\begin{equation}
( u, v )^h = \tfrac12\sum_{j=1}^J h_j
\left[(uv)(q_j^-) + (uv)(q_{j-1}^+)\right],
\label{eq:ip0}
\end{equation}
where $u(q_j^\pm)=\underset{\delta\searrow 0}{\lim}\ u(q_j\pm\delta)$.
The definition (\ref{eq:ip0}) naturally extends to vector valued functions.

In particular, we will consider fully discrete approximations of
\begin{equation} \label{eq:Hxtdiscreteh}
\left( H(x_h, x_{h,\rho}) x_{h,t}, \eta_h \right)^h
+ \left( \Phi_p(x_h, x_{h,\rho}) , \eta_{h,\rho} \right)^h 
+ \left( \Phi_z(x_h,  x_{h,\rho}) , \eta_h \right)^h = 0 \qquad
\forall\ \eta_h \in \Vh
\end{equation}
in place of \eqref{eq:Hxtdiscrete}. Using this quadrature does not affect
our derived error estimate \eqref{eq:errest}, 
as can be shown with standard techniques.

In order to discretize \eqref{eq:Hxtdiscreteh} 
in time, let $t_m=m\Delta t$, $m=0,\ldots,M$, 
with the uniform time step size $\Delta t = \frac TM >0$. 
In the following, we let $x^0_h = x_h(\cdot,0) = \pi^h x_0 \in \Vh$ 
and for $m=0,\ldots,M-1$ let $ x^{m+1}_h \in \Vh$ be the solution of a system
of algebraic equations, which we will specify. In general, we will attempt to
define fully discrete approximations that are unconditionally stable, in the
sense that they satisfy the following discrete analogue of \eqref{eq:stab},
\begin{equation} \label{eq:fdstab}
 \left( \Phi( x^k_h,  x^k_{h,\rho}) 
- \Phi( x^0_h, x^0_{h,\rho}) ,1 \right)^h \leq
-\Delta t \sum_{m = 0}^{k-1} \left( H(x^m_h,x^m_{h,\rho}) \frac{x^{m+1}_h - x^m_h}{\Delta t}, \frac{x^{m+1}_h - x^m_h}{\Delta t} \right)^h
 \leq 0,
\end{equation}
for $k = 1,\ldots,M$. Here, the second inequality is a consequence of \eqref{eq:Hposdef}.

\subsection{Space-independent anisotropic curve shortening flow}
Let $\gamma(z,p)=\gamma_0(p)$ be an anisotropy function and 
$\Phi_0(p)=\frac{1}{2} \gamma_0^2(p)$. Using Example~\ref{ex:2}\ref{item:ex2b}
we propose the scheme
\begin{equation} \label{eq:fdani}
\frac{1}{\Delta t} 
\left(H_0(x^m_{h,\rho}) ( x^{m+1}_h- x^m_h) , 
\eta_h \right)^h 
+ \left(\Phi_0'(x^{m+1}_{h,\rho}) , \eta_{h,\rho} \right)
=0 \qquad \forall\ \eta_h \in \Vh,
\end{equation}
where $H_0$ is defined in \eqref{eq:H0}. 
We remark that in the isotropic case the scheme \eqref{eq:fdani} is linear and
collapses to the fully discrete approximation 
in \cite[p.\ 108]{DeckelnickD95}.

\begin{lemma} % \label{lem:ani}
A solution $(x^m_h)_{m=0}^M$ to \eqref{eq:fdani} satisfies the stability bound \eqref{eq:fdstab}.
\end{lemma}
\begin{proof}
Choose $\eta_h =  x^{m+1}_h -  x^m_h$ in \eqref{eq:fdani},  use
\eqref{eq:lambdaL} and sum for $m = 0$ to $k-1$.
\end{proof}

A disadvantage of the scheme \eqref{eq:fdani} is that at each time level a 
nonlinear system of equations needs to be solved. Following the approach in \cite{DeckelnickD02}, see also \cite{Pozzi07}, one could alternatively consider a linear scheme 
by introducing a suitable stabilization term and treating the elliptic term
in \eqref{eq:fdani} fully explicitly.

If we restrict our attention to a special class of anisotropies, then a
linear and unconditionally stable approximation can be introduced that does 
not rely on a stabilization term. This idea goes back to \cite{triplejANI}, and was extended
to the phase field context in \cite{eck}. 
In fact, a wide class of anisotropies can either be modelled or at least very
well approximated by
\begin{equation}
\gamma_0(p) = %\sum_{\ell=1}^L \tilde\gamma_{\ell}(p) =
\sum_{\ell=1}^L \sqrt{\Lambda_{\ell} p \cdot  p} , \label{eq:BGNg}
\end{equation}
where $\Lambda_{\ell} \in \bR^{2\times 2}$, 
$\ell=1,\ldots, L$, are symmetric and positive definite, see
\cite{triplejANI,ani3d}. Hence the assumption \eqref{eq:sconv} is satisfied.
In order to be able to apply the ideas in \cite{eck,vch}, we 
define the auxiliary function $\phi_0(p)=\gamma_0(p^\perp)$, so that
$\Phi_0(p)=\frac{1}{2} \gamma^2_0(p^\perp)=\frac{1}{2} \phi_0^2(p)$.
Observe that $\phi_0$ also falls within the class of densities of the form
\eqref{eq:BGNg}, i.e.\ 
\begin{equation*} % \label{eq:phi0}
\phi_0(p)=\gamma_0(p^\perp)= \sum_{\ell=1}^L \sqrt{\widetilde \Lambda_{\ell} p \cdot  p}, \quad \mbox{ where }
\widetilde\Lambda_\ell = \det(\Lambda_\ell) \Lambda_\ell^{-1} .
\end{equation*}
Moreover, we recall from \cite{eck,vch} that $\Phi_0'(p) = B(p) p$,
if we introduce the matrices
\begin{equation} \label{eq:BGNB}
B(p) =
\begin{cases}
\gamma_0(p^\perp) \displaystyle\sum_{\ell =1}^L 
\frac{\widetilde\Lambda_\ell}{\sqrt{\widetilde\Lambda_{\ell} p \cdot p}}
& p \neq 0, \\
L \displaystyle\sum_{\ell =1}^L \widetilde\Lambda_\ell & p =0.
\end{cases}
\end{equation}
On recalling once again the definition of $H_0$ from \eqref{eq:H0},
we then consider the scheme
\begin{equation} \label{eq:fdbgn}
\frac{1}{\Delta t} \left( H_0(x^m_{h,\rho}) (x^{m+1}_h- x^m_h) ,
\eta_h \right)^h 
+ \left( B(x^m_{h,\rho})  x^{m+1}_{h,\rho} , \eta_{h,\rho} \right)
=0 \qquad \forall\ \eta_h \in \Vh,
\end{equation}
which is inspired by the treatment of the anisotropy in \cite{eck,vch} 
and leads to a system of linear equations. We note that for the case
$L=1$ the two schemes \eqref{eq:fdbgn} and \eqref{eq:fdani} are identical.

\begin{lemma} \label{lem:bgn}
Suppose that $x^m_h\in\Vh$ with 
$x^m_{h,\rho} \neq 0$ in $I$. Then there exists a unique solution $x^{m+1}_h\in\Vh$ to \eqref{eq:fdbgn}. 
Moreover, a solution $(x^m_h)_{m=0}^M$ to \eqref{eq:fdbgn} satisfies the stability bound \eqref{eq:fdstab}.
\end{lemma}
\begin{proof} Existence follows from uniqueness, and so we consider the homogeneous system:
Find $ X_h \in \Vh$ such that
\[
\frac{1}{\Delta t} \left( H_0(x^m_{h,\rho})  X_h , \eta_h
\right)^h 
+  \left( B(x^m_{h,\rho}) X_{h,\rho} , \eta_{h,\rho} \right) =0 
\qquad \forall\ \eta_h \in \Vh.
\]
Choosing $\eta_h =  X_h$, and observing that the matrices $B(p)$ 
are positive definite, we obtain
\[
0 = \left( H_0(x^m_{h,\rho})  X_h ,  X_h \right)^h 
+ \Delta t \left( B(x^m_{h,\rho}) X_{h,\rho}, X_{h,\rho} \right) 
\geq \left( H_0(x^m_{h,\rho})  X_h ,  X_h \right)^h 
\]
which implies that $X_h = 0$ in view of \eqref{eq:Hposdef}. 
This proves the existence of a unique solution. In order to show the stability bound,
we recall from Corollary 2.3 in \cite{eck} 
that with $B$ as defined in \eqref{eq:BGNB}, it holds that
\begin{equation*} % \label{eq:bqp}
B(q) p \cdot (p - q) \geq \Phi_0(p) - \Phi_0(q)
\qquad \forall\  p, q \in \bR^2.
\end{equation*}
Hence choosing $\eta_h =  x^{m+1}_h -  x^m_h$ in \eqref{eq:fdbgn} and summing 
for $m = 0$ to $k-1$ yields \eqref{eq:fdstab}.
\end{proof}

\subsection{Curve shortening flow in Riemannian manifolds}
Let us denote by $\mbox{Sym}(2,\bR)$ the set of symmetric $2 \times 2$ matrices over $\bR$. For $A,B \in \mbox{Sym}(2,\bR)$ we define
$A \succcurlyeq B$ if and only if $A-B$ is positive semidefinite. 
We say that a differentiable function $f: \Omega \rightarrow \mbox{Sym}(2,\bR)$ is convex if
\begin{displaymath}
f(w) \succcurlyeq f(z) + (w_i - z_i) f_{z_i}(z) \quad \mbox{ for all }  z, w
\in \Omega \mbox{ such that } [z,w] \subset \Omega. 
\end{displaymath}

We now consider the situation in Example~\ref{ex:2}\ref{item:ex2c}.
Let $G: \Omega \rightarrow \mbox{Sym}(2,\bR)$, so that $\Phi(z,p)=\frac{1}{2}
G(z) p \cdot p$, where $G(z)$ is positive definite for $z \in \Omega$.
In order to obtain an unconditionally stable scheme, we adapt an idea from 
\cite{hypbol} and assume that we can split $G$ into 
\begin{equation} \label{eq:splitG}
G=G_+ + G_- \quad \mbox{ such that } % z \mapsto \pm G_{\pm}(z) 
\pm G_\pm : \Omega \rightarrow \mbox{Sym}(2,\bR) \mbox{ are convex}.
\end{equation}
Such a splitting exists if there exists a constant $c_G \in \bRgeq$ such that
\begin{displaymath}
\tfrac12 \lambda_i \lambda_j G_{z_i z_j}(z) + c_G |\lambda|^2\Id \succcurlyeq 0 
\qquad \forall\ z \in \Omega,\ \lambda \in \bR^2.
\end{displaymath}
In that case one may choose $G_+(z)=G(z) + c_G |z|^2\Id$ and 
$G_-(z) = -c_G |z|^2\Id$.
It follows from (\ref{eq:splitG}) that
\begin{equation} \label{eq:splitGstab}
(w_i - z_i) \bigl( G_{+,z_i}(w) + G_{-,z_i}(z) \bigr) \succcurlyeq G(w) - G(z) 
\qquad \forall\  w, z \in \Omega\
\mbox{ such that } [z,w] \subset \Omega.
\end{equation}
We now consider the scheme
\begin{align} \label{eq:fdriem}
& \frac{1}{\Delta t} \left( H(x^m_h,x^m_{h,\rho}) 
( x^{m+1}_h- x^m_h) , \eta_h \right)^h 
+ \left( G(x^m_h)  x^{m+1}_{h,\rho} , \eta_{h,\rho} \right)^h 
\nonumber \\ & \qquad 
+ \tfrac{1}{2} \left( 
\eta_{h,i} \bigl( G_{+,z_i}(x^{m+1}_h) + G_{-,z_i}( x^m_h) \bigr) 
x^{m+1}_{h,\rho} ,  x^{m+1}_{h,\rho} \right)^h =0 
\qquad \forall\ \eta_h \in \Vh,
\end{align}
where $H$ is as defined in \eqref{eq:Hriem}. We note that in general 
\eqref{eq:fdriem} is a nonlinear scheme. 

\begin{lemma} % \label{lem:fdriem}
Let $(x^m_h)_{m=0}^M$ be a solution to
\eqref{eq:fdriem},
with $[x^m_h(q_j), x^{m+1}_h(q_j)] \subset \Omega$ for 
$j=1,\ldots,J$ and $m=0,\ldots,M-1$. Then 
the stability bound \eqref{eq:fdstab} is satisfied.
\end{lemma}
\begin{proof}
Let us again choose $\eta_h = x^{m+1}_h - x^m_h$ in \eqref{eq:fdriem} and calculate
with the help of \eqref{eq:splitGstab}
\begin{align*}
& \left( G( x^m_h)  x^{m+1}_{h,\rho}, 
 x^{m+1}_{h,\rho} - x^m_{h,\rho} \right)^h  + \tfrac{1}{2} \left( (x^{m+1}_{h,i} - x^m_{h,i}) 
\bigl( G_{+,z_i}( x^{m+1}_h) + G_{-,z_i}(x^m_h) \bigr) 
 x^{m+1}_{h,\rho} ,  x^{m+1}_{h,\rho} \right)^h \\ & \quad
\geq  \left( G(x^m_h)  x^{m+1}_{h,\rho} ,  x^{m+1}_{h,\rho} \right)^h
- \left( G(x^m_h)  x^{m+1}_{h,\rho} ,  x^m_{h,\rho} \right)^h 
+ \tfrac{1}{2} \left( \bigl( G(x^{m+1}_h) - G(x^m_h) \bigr) 
 x^{m+1}_{h,\rho} ,  x^{m+1}_{h,\rho} \right)^h \\ & \quad
=  \tfrac12 \left ( G(x^m_h) (x^{m+1}_{h,\rho} -  x^m_{h,\rho}) ,
 x^{m+1}_{h,\rho} -  x^m_{h,\rho} \right)^h 
+ \left(\Phi(x^{m+1}_h,x^{m+1}_{h,\rho}) ,1 \right)^h 
- \left(\Phi(x^m_h,x^m_{h,\rho}) , 1 \right)^h  \\ & \quad 
\geq  \left(\Phi(x^{m+1}_h, x^{m+1}_{h,\rho}) ,1 \right)^h 
- \left(\Phi(x^m_h, x^m_{h,\rho}) , 1 \right)^h,
\end{align*}
since $G(x^m_h)$ is symmetric and positive definite.
This yields the desired result similarly to the proof of Lemma~\ref{lem:bgn}.
\end{proof}

In the special case that the Riemannian manifold is conformally equivalent 
to the Euclidean plane, i.e.\ when $G(z) = \altg(z) \Id$
for $\altg : \Omega \to \bRplus$, several
numerical schemes have been proposed in \cite[\S3.1]{hypbol}. 
In this situation our fully discrete approximation \eqref{eq:fdriem} 
collapses to the new scheme
\begin{align} \label{eq:fdhypbol}
& \frac1{\Delta t} \left( \altg^2(x_h^m) 
( x^{m+1}_h -  x^m_h) , \eta_h | x^m_{h,\rho}|^2 \right)^h 
+ \left( \altg(x_h^{m}) x^{m+1}_{h,\rho}, \eta_{h,\rho} \right)^h
\nonumber \\ & \qquad
+ \tfrac12 \left( (\nabla \altg_+ (x_h^{m+1}) + 
\nabla \altg_- (x_h^{m})), \eta_h | x^{m+1}_{h,\rho}|^2 \right)^h
= 0 \qquad \forall\ \eta_h \in \Vh, 
\end{align}
where $\altg=\altg_+ + \altg_-$ and $\pm \altg_\pm: \Omega \rightarrow \bR$ are convex.
We observe that for the special case $\altg(z) = (z_1)^2$ the scheme 
\eqref{eq:fdhypbol} is in fact very close to the approximation 
\cite[(4.4)]{schemeD}, 
modulo the different time scaling factor that arises in the context of
mean curvature flow for axisymmetric hypersurfaces in $\bR^3$ considered there.

\setcounter{equation}{0}
\section{Numerical results} \label{sec:nr}

We implemented our fully discrete schemes within the
finite element toolbox Alberta, \cite{Alberta}. Where the systems of 
equations arising at each time level are nonlinear, they
are solved using a Newton method or a Picard-type iteration,
while all linear (sub)problems are solved
with the help of the sparse factorization package UMFPACK, see \cite{Davis04}.
For example, for the solution of \eqref{eq:fdani} we employ a Newton iteration,
while the Picard iteration for the solution of
\eqref{eq:fdriem} is defined through $x^{m+1,0}_h = x^m_h$ and, for $i \geq 0$,
by $x^{m+1,i+1}_h \in \Vh$ such that
\begin{align} \label{eq:Picard}
& \frac{1}{\Delta t} \left( H(x^m_h,x^m_{h,\rho}) 
( x^{m+1,i+1}_h- x^m_h) , \eta_h \right)^h 
+ \left( G(x^m_h)  x^{m+1,i+1}_{h,\rho} , \eta_{h,\rho} \right)^h 
\nonumber \\ & \qquad 
+ \tfrac{1}{2} \left( 
\eta_{h,i} \bigl( G_{+,z_i}(x^{m+1,i}_h) + G_{-,z_i}( x^m_h) \bigr) 
x^{m+1,i}_{h,\rho} ,  x^{m+1,i}_{h,\rho} \right)^h =0 
\qquad \forall\ \eta_h \in \Vh.
\end{align}
In all our simulations the Newton solver for \eqref{eq:fdani} converged in at
most one iteration, while the Picard iteration \eqref{eq:Picard} always
converged in at most three iterations.

For all our numerical simulations we use a uniform partitioning of $[0,1]$,
so that $q_j = jh$, $j=0,\ldots,J$, with $h = \frac 1J$. Unless otherwise
stated, we use $J=128$ and $\Delta t = 10^{-4}$.
On recalling (\ref{eq:aniso}), for $\chi_h \in \Vh$ we define the discrete
energy
\begin{equation*} % \label{eq:Lgh}
\mathcal{E}^{h}(\chi_h) = 
\left(\altm(\chi_h), \gamma(\chi_h, \chi_{h,\rho}^\perp)\right)^{h} .
\end{equation*}
We also consider the ratio
\begin{equation} \label{eq:ratio}
\ratio^m = \dfrac{\max_{j=1,\ldots,J} |x_h^m(q_j) - x_h^m(q_{j-1})|}
{\min_{j=1,\ldots, J} |x_h^m(q_j) - x_h^m(q_{j-1})|}
\end{equation}
between the longest and shortest element of $\Gamma^m_h = x_h^m(I)$, 
and are often interested in the evolution of this ratio over time.
We stress that no redistribution of vertices was necessary during any of our
numerical simulations.
In the isotropic case this can be explained by the diffusive character of the
tangential motion induced by \eqref{eq:csfdeTurck}, since the flow can be
rewritten as $x_t = \varkappa\nu - (\frac1{|x_\rho|})_\rho\tau$,
as has been pointed out in e.g.\ \cite[p.~1477]{MikulaS01}. 
Our numerical experiments indicate that while the induced tangential motion
from \eqref{eq:Hxt} may in general not be diffusive, it is sufficiently
benevolent to avoid coalescence of vertices in practice.

\subsection{Space-independent anisotropic curve shortening flow}
In this subsection, we consider the situation from
Example~\ref{ex:1}\ref{item:ex1b}, see also 
Example~\ref{ex:2}\ref{item:ex2b}. 
Anisotropies of the form $\gamma(z,p) = \gamma_0(p)$
can be visualized by their Frank diagram
$\mathcal F = \{ p\in\bR^2 : \gamma_0(p) \leq 1 \}$
and their Wulff shape
$\mathcal W = \{ q\in\bR^2 : \gamma^*_0(q) \leq 1 \}$,
where $\gamma^*_0$ is the dual to $\gamma_0$ defined by 
\begin{equation*}
\gamma^*_0(q) = \sup_{p\in\bR^2\setminus\{0\}} \frac{p\cdot q}{\gamma_0(p)}.
\end{equation*}
We recall from \cite{FonsecaM91} that the boundary of the Wulff shape,
$\partial\mathcal W$, is the solution of the isoperimetric
problem for $\mathcal E(\Gamma) = \int_\Gamma \gamma_0(\nu) \dH1$.
Moreover, it was shown in \cite{Soner93} that self-similarly
shrinking boundaries of Wulff 
shapes are a solution to anisotropic curve shortening flow. In particular,
\begin{equation}
\Gamma(t) = \{ q \in \bR^2 : \gamma^*_0(q) = \sqrt{1 - 2t} \}
\label{eq:Wexact}
\end{equation}
solves \eqref{eq:acsf0}. 
We demonstrate this behaviour in Figure~\ref{fig:anibgnL1equid} for the
``elliptic'' anisotropy
\begin{equation} \label{eq:gamma0L1}
\gamma_0(p) = \sqrt{p_1^2 + \delta^2 p_2^2}, \quad \delta = 0.5.
\end{equation}
Observe that \eqref{eq:gamma0L1} is a special case of \eqref{eq:BGNg} 
with $L=1$, so that the scheme \eqref{eq:fdani} 
collapses to the linear scheme \eqref{eq:fdbgn}. 
The evolution in Figure~\ref{fig:anibgnL1equid} nicely shows how the curve 
shrinks self-similarly to a point.
We can also see that the scheme \eqref{eq:fdani} induces a
tangential motion that moves the initially equidistributed vertices along
the curve, so that eventually a higher density of vertices can be
observed in regions of larger curvature. We note that this behaviour is 
not dissimilar to the behaviour observed in the numerical experiments in 
\cite{triplejANI}. 

We now use the exact solution \eqref{eq:Wexact} to perform a convergence
experiment for our proposed finite element approximation 
\eqref{eq:Hxtdiscrete}. To this end, we choose the particular parameterization
\begin{equation} \label{eq:true}
x(\rho,t)=(1-2t)^{\frac12} \binom{\cos(2\pi\rho)}{\delta \sin(2\pi\rho)}
\end{equation}
and define
\begin{equation} \label{eq:f}
f = H_0(x_\rho) x_t - [\Phi_0'(x_\rho)]_\rho ,
\end{equation}
so that \eqref{eq:true} is the exact solution of \eqref{eq:ex2b} with the
additional right hand side $(f, \eta)$. Upon adding the right hand side
$(f(t_{m+1}), \eta_h)^h$ to \eqref{eq:fdani}, 
we can thus use
\eqref{eq:true} as a reference solution for a convergence experiment of our
proposed finite element approximation.
We report on the observed $H^1$-- and $L^2$--errors
for the scheme \eqref{eq:fdani} for a sequence of mesh sizes
in Table~\ref{tab:aMCtrue}.
Here we partition the time interval $[0,T]$, with $T=0.45$, into uniform
time steps of size $\Delta t = h^2$, for $h = J^{-1} = 2^{-k}$, 
$k=5,\ldots,9$. The observed numerical results confirm the optimal convergence 
rate for the $H^1$--error from Theorem~\ref{thm:errest}.
\begin{figure}
\center
\includegraphics[angle=-90,width=0.45\textwidth]{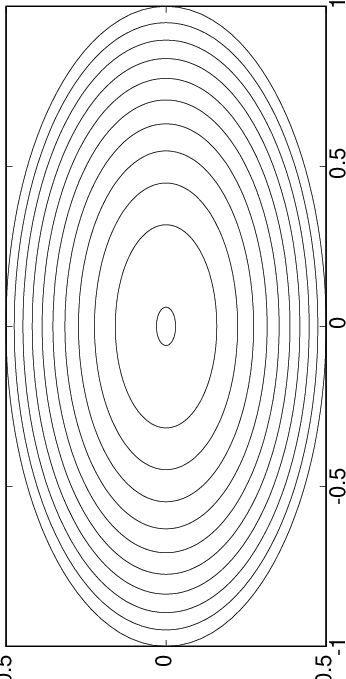}\qquad
\includegraphics[angle=-90,width=0.4\textwidth]{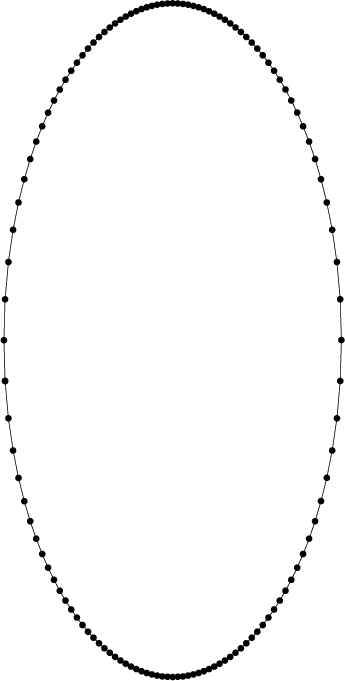}
\includegraphics[angle=-90,width=0.35\textwidth]{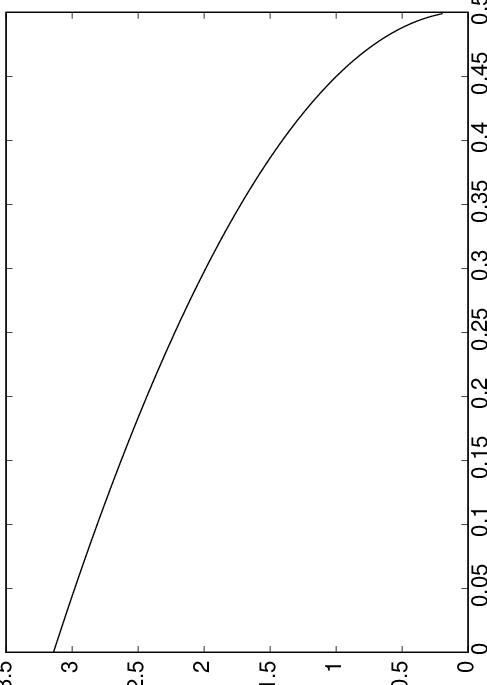}
\includegraphics[angle=-90,width=0.35\textwidth]{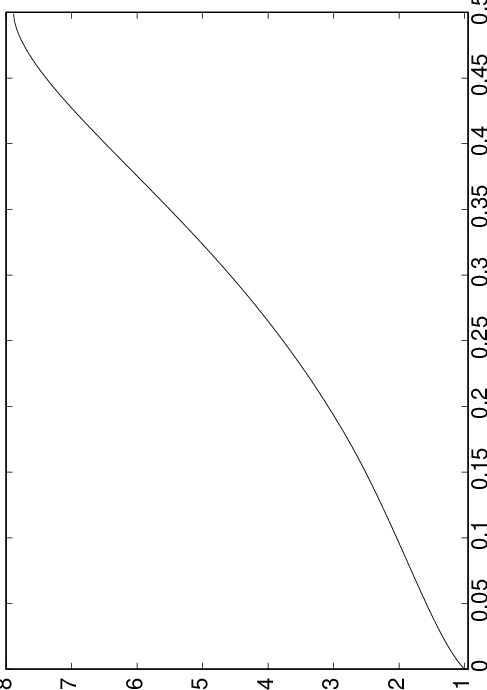}
\caption{Anisotropic curvature flow for \eqref{eq:gamma0L1} using the
scheme \eqref{eq:fdani}.
Solution at times $t=0,0.05,\ldots,0.45,0.499$, as well as the
distribution of vertices on $x_h^M(I)$.
Below we show a plot of the discrete energy $\mathcal{E}^h(x_h^m)$
and of the ratio (\ref{eq:ratio}) over time.} 
\label{fig:anibgnL1equid}
\end{figure}
\begin{table}
\center
\begin{tabular}{|r|c|c|c|c|}
\hline
$J$ & $\displaystyle\max_{m=0,\ldots,M} \| x(\cdot,t_m) -  x^m_h\|_0$ & EOC
& $\displaystyle\max_{m=0,\ldots,M} \| x(\cdot,t_m) -  x^m_h\|_1$ & EOC 
\\ \hline
32  & 1.2337e-02 & ---  & 2.8140e-01 & ---  \\
64  & 3.1870e-03 & 1.95 & 1.4076e-01 & 1.00 \\
128 & 8.0360e-04 & 1.99 & 7.0386e-02 & 1.00 \\
256 & 2.0133e-04 & 2.00 & 3.5194e-02 & 1.00 \\
512 & 5.0361e-05 & 2.00 & 1.7597e-02 & 1.00 \\
\hline
\end{tabular}
\caption{Errors for the convergence test for (\ref{eq:true})
over the time interval $[0,0.45]$ for the scheme \eqref{eq:fdani} 
with the additional right hand side $(f(t_{m+1}), \eta_h)^h$ from \eqref{eq:f}.
We also display the experimental orders of convergence (EOC).}
\label{tab:aMCtrue}
\end{table}%

Next we consider smooth anisotropies as in \cite[(7.1)]{Dziuk99}
and \cite[(4.4a)]{fdfi}. To this end, let
\begin{equation} \label{eq:gammaD99b}
\gamma_0( p) = | p| (1 + \delta \cos(k\theta(p))),
\quad  p = | p| \binom{\cos\theta(p)}{\sin\theta(p)},\quad 
k \in \bN,\ \delta \in \bRgeq.
\end{equation}
It is not difficult to verify that this anisotropy satisfies 
\eqref{eq:sconv} if and only if $\delta < \frac1{k^2 - 1}$,
see also \cite[p.~27]{fdfi}.
In order to visualize this family of anisotropies, we show a
Frank diagram and a Wulff shape for the cases $k=3$ and $k=6$, respectively, in 
Figure~\ref{fig:frankD99b}.
\begin{figure}
\center
\includegraphics[clip,angle=-0,width=0.2\textwidth]{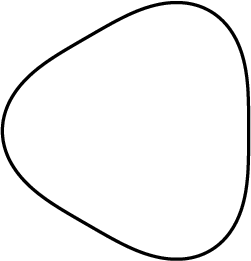}\quad
\includegraphics[clip,angle=-0,width=0.2\textwidth]{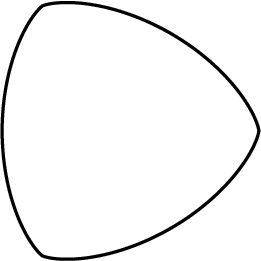}\qquad\qquad
\includegraphics[clip,angle=-0,width=0.2\textwidth]{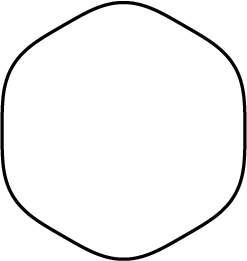}\quad
\includegraphics[clip,angle=-0,width=0.2\textwidth]{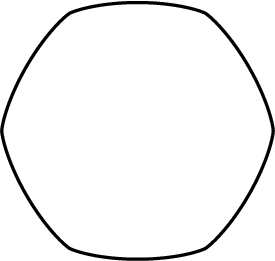}
\caption{Frank diagram (left) and Wulff shape (right) for 
\eqref{eq:gammaD99b} with $(k,\delta)=(3, 0.124)$ and $(6,0.028)$.}
\label{fig:frankD99b}
\end{figure}%
We show the evolutions for anisotropic curve shortening flow induced by 
these two anisotropies, starting in each case from an equidistributed 
approximation of a unit circle, in Figure~\ref{fig:ani2D99b}. 
Here we use the fully discrete scheme \eqref{eq:fdani}. 
We observe from the evolutions in Figure~\ref{fig:ani2D99b} that the shape of 
the curve quickly approaches the Wulff shape, while it continuously shrinks 
towards a point. 
It is interesting to note that the ratio \eqref{eq:ratio} increases only
slightly and then appears to remain nearly constant for the remainder of the
evolution.
\begin{figure}
\center
\mbox{
\includegraphics[clip,angle=-90,width=0.21\textwidth]{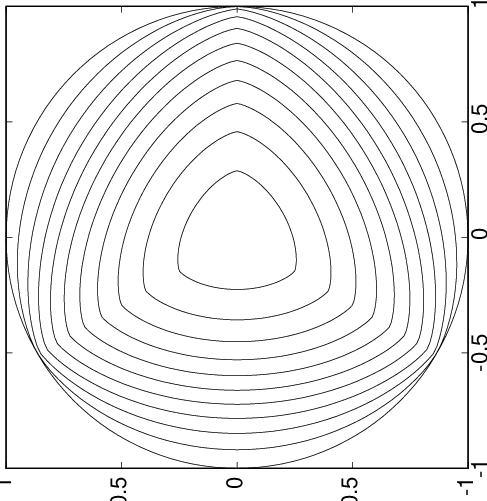}
\includegraphics[angle=-90,width=0.28\textwidth]{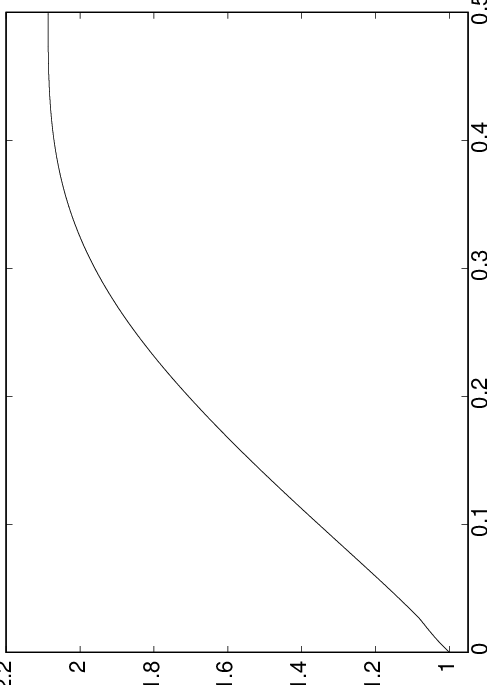}
\includegraphics[angle=-90,width=0.21\textwidth]{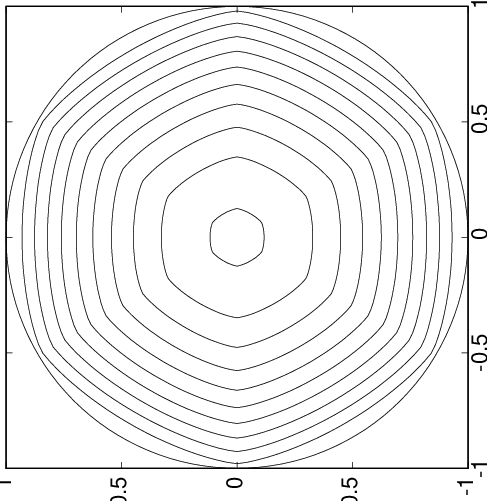}
\includegraphics[angle=-90,width=0.28\textwidth]{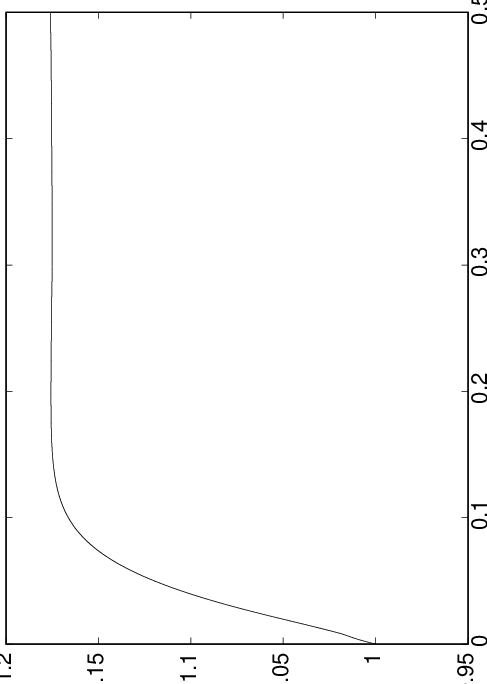}}
\caption{Anisotropic curvature flow for \eqref{eq:gammaD99b},
with $(k,\delta)=(3, 0.124)$ (left) and $(6,0.028)$ (right), using the 
scheme \eqref{eq:fdani}. 
Solution at times $t=0,0.05,\ldots,0.5$.
We also show plots %of the discrete energy $\mathcal{E}^h( x_h^m)$ and 
of the ratio (\ref{eq:ratio}) over time.} 
\label{fig:ani2D99b}
\end{figure}

One motivation for choosing a sixfold anisotropy, as in \eqref{eq:gammaD99b} 
with $k=6$, is its relevance for modelling ice crystal growth, 
see e.g.\ \cite{dendritic,jcg}. Here it is desirable to choose a (nearly)
crystalline anisotropy, which means that the Wulff shape exhibits flat sides
and sharp corners.
With the help of the class of anisotropies \eqref{eq:BGNg}
this is possible. We immediately demonstrate how this can be achieved for
a general $k$-fold symmetry, for even $k\in\bN$. On choosing $L = k/2$,
we define the rotation matrix
$Q(\theta)=\binom{\phantom{-}\cos\theta\ \sin\theta}{-\sin\theta\ \cos\theta}$
and the diagonal matrix $D(\delta) = \diag(1,\delta^2)$, and then let
\begin{equation} \label{eq:gammabgnL}
\gamma_0(p) 
= \sum_{\ell=1}^L \sqrt{ [(Q(\tfrac\pi{L})^\ell]^T D(\delta) 
(Q(\tfrac\pi{L}))^\ell p \cdot p}, \quad \delta \in \bRplus.
\end{equation}
We visualize some Wulff shapes of \eqref{eq:gammabgnL} for $L=2,3,4$ in
Figure~\ref{fig:wulffL234} and observe that these Wulff shapes, for
$\delta\to0$, will approach a square, a regular hexagon and a regular octagon,
respectively.
\begin{figure}
\center
\includegraphics[clip,width=0.2\textwidth]{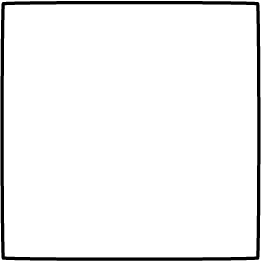}\qquad
\includegraphics[clip,width=0.2\textwidth]{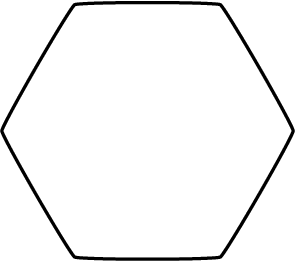}\qquad
\includegraphics[clip,width=0.2\textwidth]{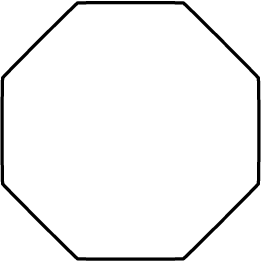}
\caption{Wulff shapes (scaled) for \eqref{eq:gammabgnL} with $L=2,3,4$ and
$\delta=10^{-2}$.}
\label{fig:wulffL234}
\end{figure}%
Of course, the associated crystalline anisotropic energy densities,
when $\delta = 0$, are no
longer differentiable, and so the theory developed in this paper no longer
applies.
Yet, for a fixed $\delta>0$ all the assumptions in this paper are
satisfied and our scheme \eqref{eq:fdbgn} works extremely well.
As an example, we repeat the simulations in Figure~\ref{fig:ani2D99b} for the
anisotropy \eqref{eq:gammabgnL} with $L=2$ and $\delta=10^{-2}$, now
using the scheme \eqref{eq:fdbgn}.  
{From} the evolution shown in Figure~\ref{fig:anibgn_L2} it can be seen that
the initial curve assumes the shape of a smoothed square that then shrinks to
a point. We also observe that
after an initial increase the ratio \eqref{eq:ratio} decreases and
eventually reaches a steady state. The final
distribution of mesh points is such that there is a slightly
lower density of vertices on the nearly flat parts of the curve.
\begin{figure}
\center
\includegraphics[clip,angle=-90,width=0.3\textwidth]{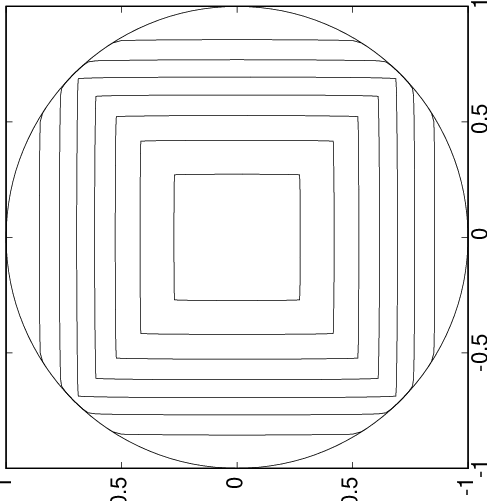}\quad
\includegraphics[clip,angle=-90,width=0.25\textwidth]{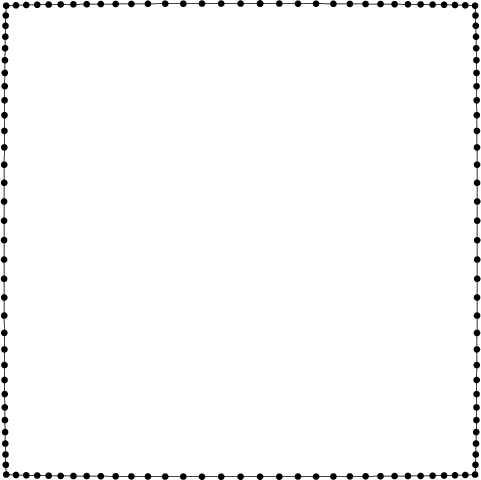}\qquad
\includegraphics[angle=-90,width=0.35\textwidth]{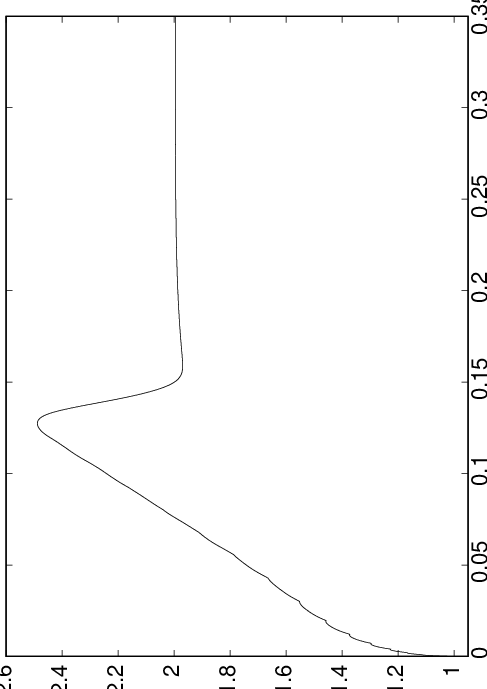}
\caption{Anisotropic curvature flow for \eqref{eq:gammabgnL} with 
$L = 2$ and $\delta = 10^{-2}$, using the scheme \eqref{eq:fdbgn}.
We show the solution at times $t=0,0.05,\ldots,0.35$, the
distribution of vertices on $x_h^M(I)$, as well as the
evolution of the ratio (\ref{eq:ratio}) over time.}
\label{fig:anibgn_L2}
\end{figure}%

Inspired by the computations in \cite[Fig.\ 6.1]{ObermanOTT11} we now consider
evolutions for an initial curve that consists of a $\frac32\pi$-segment
of the unit circle merged with parts of a square of side length 2. 
For our computations we employ the scheme \eqref{eq:fdbgn},
with the discretization parameters $J=256$ and $\Delta t = 10^{-4}$.
The evolutions for the three anisotropies visualized in
Figure~\ref{fig:wulffL234} can be seen in Figure~\ref{fig:OOTT11_bgnL234}.
We observe that the smooth part of the initial curve transitions
into a crystalline shape, while the initial facets of the curve that are 
aligned with the Wulff shape simply shrink.
The other facets disappear, some immediately and some over time, 
as they are replaced by facets aligned with the Wulff shape.
Particularly interesting is the evolution of the left nonconvex corner in the 
initial curve, which shows three qualitatively very different behaviours for 
the three chosen anisotropies.
\begin{figure}
\center
\includegraphics[angle=-90,width=0.3\textwidth]{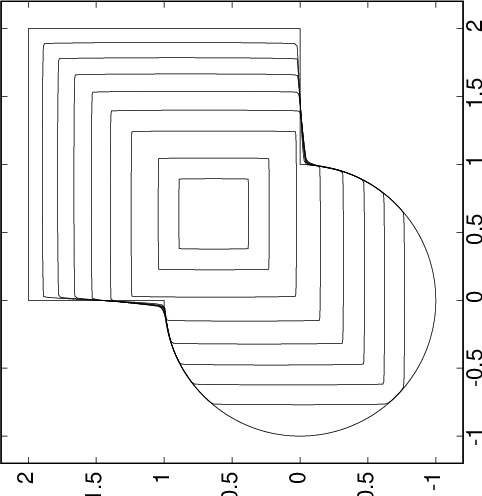}\quad
\includegraphics[angle=-90,width=0.3\textwidth]{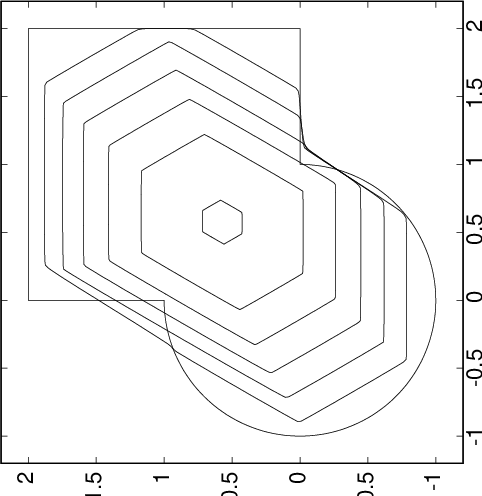}\quad
\includegraphics[angle=-90,width=0.3\textwidth]{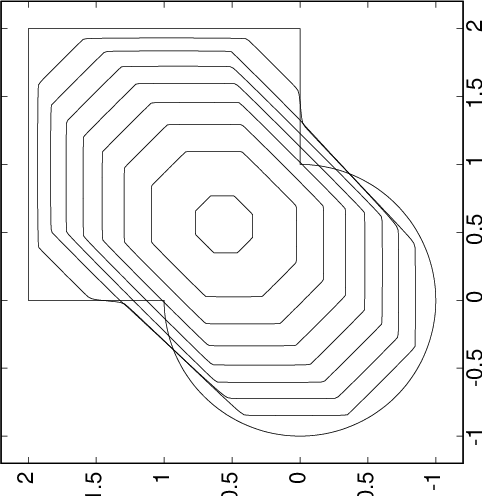}
\caption{Anisotropic curvature flow for \eqref{eq:gammabgnL}
with $L=2,3,4$ and $\delta=10^{-2}$, using the scheme \eqref{eq:fdbgn}.
We show the solution at times $t=0,0.1,\ldots,0.7,0.75$ (left),
$t=0,0.05,\ldots,0.3$ (middle) and $t=0,0.02,\ldots,0.16$ (right).
} 
\label{fig:OOTT11_bgnL234}
\end{figure}%

For the final simulations in this subsection we choose as initial data a
polygon that is very similar to the initial curve from 
\cite[Fig.~0]{AlmgrenT95}.
In their seminal work, Almgren and Taylor consider motion by crystalline
curvature, which is the natural generalization of anisotropic curve shortening
flow to purely crystalline anisotropies, that is when the Wulff shape is a
polygon. For motion by crystalline curvature
a system of ODEs for the sizes and positions of all the facets
of an evolving polygonal curve has to be solved. Here the initial curve needs
to be admissible, in the sense that it only exhibits facets that also appear 
in the Wulff shape, and any two of its neighbouring facets are
also neighbours in the Wulff shape.
Hence the initial curve for the computations shown in
Figure~\ref{fig:AlmgrenT95_bgnL234}, for which we employed the
scheme \eqref{eq:fdbgn} with $J=512$ and $\Delta t = 10^{-4}$,
is admissible for an eightfold anisotropy, 
with a regular octagon as Wulff shape. 
Our simulation for 
the smoothed anisotropy \eqref{eq:gammabgnL} with $L=4$ and $\delta=10^{-4}$ 
agrees remarkably well with the evolution shown in \cite[Fig.\ 0]{AlmgrenT95}.
In fact, it is natural to conjecture that in the
limit $\delta\to0$, anisotropic curve shortening flow for the anisotropies
\eqref{eq:gammabgnL} converges to flow by crystalline curvature
with respect to the crystalline surface energies \eqref{eq:gammabgnL} with
$\delta=0$.
\begin{figure}
\center
\includegraphics[angle=-90,width=0.3\textwidth]{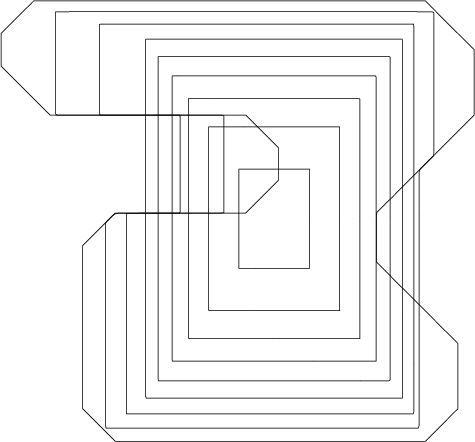}\qquad
\includegraphics[angle=-90,width=0.3\textwidth]{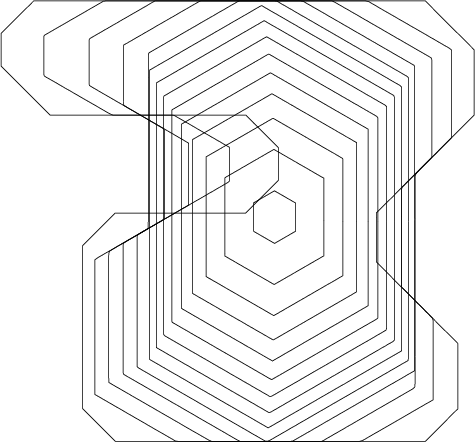}\qquad
\includegraphics[angle=-90,width=0.3\textwidth]{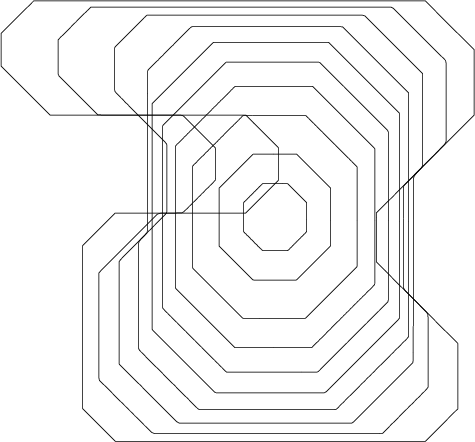}
\caption{Anisotropic curvature flow for \eqref{eq:gammabgnL}
with $L=2,3,4$ and $\delta=10^{-4}$, using the scheme \eqref{eq:fdbgn}.
We show the solution at times $t=0,2,\ldots,16$ (left),
$t=0,0.5,\ldots,6,6.4$ (middle) and $t=0,0.4,\ldots,3.2,3.4$ (right).
} 
\label{fig:AlmgrenT95_bgnL234}
\end{figure}%
We stress that for the cases $L=2$ and $L=3$, when the Wulff shape is
a square and a regular hexagon, respectively, the initial curve
in Figure~\ref{fig:AlmgrenT95_bgnL234} is no longer admissible in the
sense described above. 
As we only deal with the case $\delta>0$, our scheme \eqref{eq:fdbgn} has no
difficulties in computing the evolutions shown in 
Figure~\ref{fig:AlmgrenT95_bgnL234} for $L=2$ and $L=3$.
We observe once again that new facets appear where the initial polygon is not 
aligned with the Wulff shape, while the admissible facets simply shrink.

\subsection{Curve shortening flow in Riemannian manifolds}

In this subsection we consider the setup from
Example~\ref{ex:1}\ref{item:ex1c}, see also 
Example~\ref{ex:2}\ref{item:ex2c}. 
At first we look at the simpler case of a manifold that is conformally flat, 
so that we can employ the scheme \eqref{eq:fdhypbol}. 
As an example we take $G(z) = \altg(z) \Id$ with $\altg(z) = (z_1)^{-2}$
and note that with $\Omega = \{ z \in \bR^2 : z_1 > 0 \}$
we obtain a model for the hyperbolic plane, which is a 
two-dimensional manifold that cannot be embedded into $\bR^3$,
as was proved by Hilbert, \cite{Hilbert1901}, see also
\cite[\S11.1]{Pressley10}.
{From} \cite[Appendix~A]{hypbol}, and on noting
Lemma~\ref{lem:appB} in Appendix~\ref{sec:appB}, 
we recall that a true solution for
\eqref{eq:acsf}, i.e.\ geodesic curvature flow in the hyperbolic plane,
is given by a family of translating and shrinking circles in $\Omega$:
\begin{equation} \label{eq:truehypbol}
\Gamma(t) = \tbinom{a(t)}{0} + r(t) \bS^1,\quad
a(t) = e^{-t}a(0),\
r(t) = \left( r^2(0) - a^2(0) \left[ 1 - e^{-2t} \right]\right)^\frac12,
\end{equation}
with $a(0) > r(0) > 0$ and $\bS^1 = \{ z \in \bR^2 : |z|=1\}$.
In Figure~\ref{fig:flatmu1_circle} we show such an evolution,
starting from a unit circle centred at $\binom{2}{0}$, 
computed with the scheme \eqref{eq:fdhypbol},
where, since $\altg$ is convex in $\Omega$, we choose $\altg_+ = \altg$. 
We observe that during the evolution the discrete geodesic length is
decreasing, while the approximation to the
shrinking circle remains nearly equidistributed throughout.
At the final time $T=0.14$ the maximum difference between $r(T)$ and
$|x_h^M(q_j) - \tbinom{a(T)}{0}|$, for $1\leq j\leq J$, is less than
$6\cdot10^{-3}$, indicating that the polygonal curve $\Gamma_h^M = x_h^M(I)$ is
a very good approximation of the true solution $\Gamma(T)$ from
\eqref{eq:truehypbol}. 
\begin{figure}
\center
\includegraphics[angle=-90,width=0.25\textwidth]{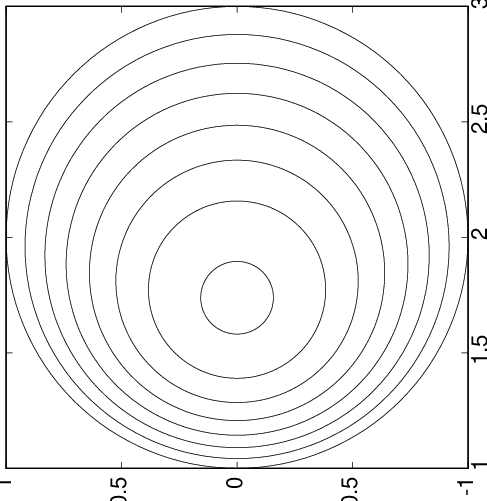}\quad
\includegraphics[angle=-90,width=0.35\textwidth]{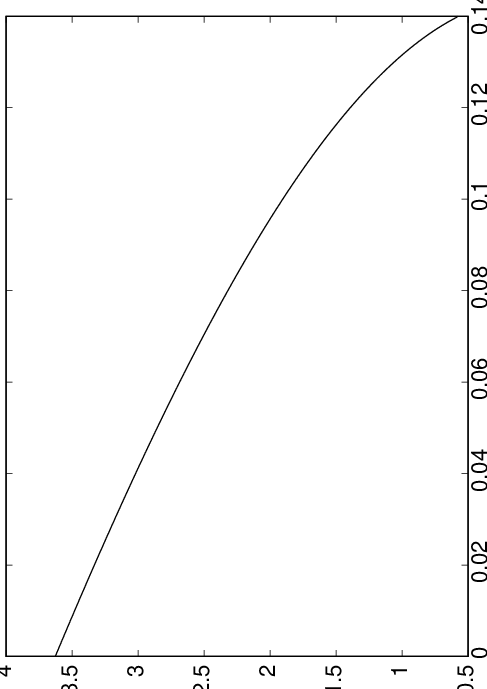}
\includegraphics[angle=-90,width=0.35\textwidth]{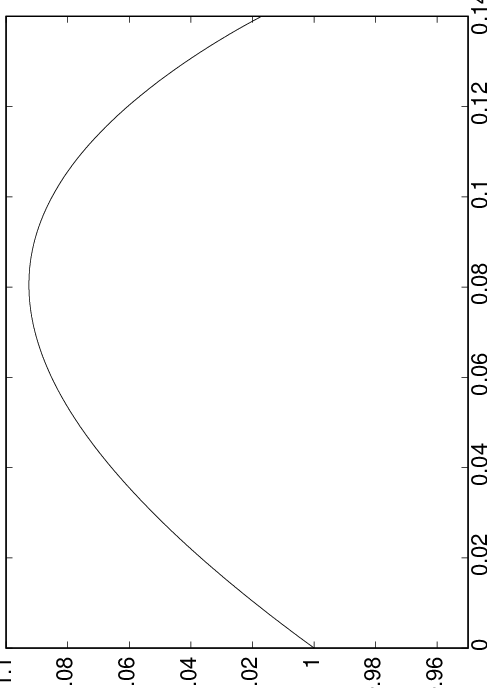}
\caption{Curvature flow in the hyperbolic plane, using the scheme
\eqref{eq:fdhypbol}. 
Solution at times $t=0,0.02,\ldots,0.14$.
We also show plots of the discrete energy $\mathcal{E}^h(x_h^m)$
and of the ratio (\ref{eq:ratio}) over time.} 
\label{fig:flatmu1_circle}
\end{figure}

For the remainder of this subsection we consider general 
Riemannian manifolds that are not necessarily
conformally flat. An example application is the modelling of geodesic
curvature flow on a hypersurface in $\bR^3$ that is given by a graph.
In particular, we assume that
\begin{equation} \label{eq:Fvarphi}
F(z) = (z_1, z_2, \varphi(z))^T,\quad \varphi \in C^3(\Omega).
\end{equation} 
The induced matrix $G$ is then given by 
$G(z) = \Id + \nabla \varphi(z) \otimes \nabla \varphi(z)$,
and the splitting \eqref{eq:splitG} for the scheme \eqref{eq:fdriem} can
be defined by $G_+(z)=G(z) + c_\varphi |z|^2 \Id$ and 
$G_-(z)=- c_\varphi |z|^2 \Id$, with
$c_\varphi \in \bRgeq$ chosen sufficiently large. 
In all our computations
we observed a monotonically decreasing discrete energy when choosing
$c_\varphi=0$, and so we always let $G_+ = G$.

We begin with a convergence experiment on the right circular cone 
defined by $\varphi(z) = b |z|$ and
$\Omega = \bR^2 \setminus\{0\}$ in \eqref{eq:Fvarphi}, for some 
$b \in \bRgeq$. A simple calculation verifies that the family of curves
$\tilde\Gamma(t) = r(t) (\bS^1 \times \{ b \})$, with
$r(t) = [r^2(0) - \frac{2t}{1+b^2}]^\frac12$ and $r(0) > 0$,
evolves under geodesic curvature flow on $\mathcal M = F(\Omega)$.
In fact, it is not difficult to show that the particular parameterization
\begin{equation} \label{eq:riemconetrue}
x(\rho,t)= [r^2(0) - \frac{2t}{1+b^2}]^\frac12
\binom{\cos(2\pi\rho)}{\sin(2\pi\rho)},
\end{equation}
so that $\tilde\Gamma(t) = F(x(I,t))$, solves \eqref{eq:Hxt}.
Similarly to Table~\ref{tab:aMCtrue}, we report on the 
$H^1$-- and $L^2$--errors between \eqref{eq:riemconetrue}, for $b = \sqrt{3}$
and $r(0)=1$, and the discrete
solutions for the scheme \eqref{eq:fdriem} in Table~\ref{tab:riemconetrue}.
Here for a sequence of mesh sizes we use uniform time steps of size 
$\Delta t = h^2$, for $h = J^{-1} = 2^{-k}$, $k=5,\ldots,9$. 
Once again, the observed numerical results confirm the optimal convergence 
rate from Theorem~\ref{thm:errest}.
\begin{table}
\center
\begin{tabular}{|r|c|c|c|c|}
\hline
$J$ & $\displaystyle\max_{m=0,\ldots,M} \| x(\cdot,t_m) -  x^m_h\|_0$ & EOC
& $\displaystyle\max_{m=0,\ldots,M} \| x(\cdot,t_m) -  x^m_h\|_1$ & EOC 
\\ \hline
32  & 1.6096e-02 & ---  & 3.5595e-01 & ---  \\
64  & 4.2080e-03 & 1.94 & 1.7805e-01 & 1.00 \\
128 & 1.0635e-03 & 1.98 & 8.9032e-02 & 1.00 \\
256 & 2.6656e-04 & 2.00 & 4.4517e-02 & 1.00 \\
512 & 6.6685e-05 & 2.00 & 2.2259e-02 & 1.00 \\
1024& 1.6674e-05 & 2.00 & 1.1129e-02 & 1.00 \\
\hline
\end{tabular}
\caption{Errors for the convergence test for (\ref{eq:riemconetrue}),
with $b = \sqrt{3}$ and  $r(0) = 1$, 
over the time interval $[0,\frac12]$ for the scheme 
\eqref{eq:fdriem} with $G_+ = G$.}
\label{tab:riemconetrue}
\end{table}%

On the same cone $\mathcal M$, we perform two computations for a curve evolving
by geodesic curvature flow.
For the simulation on the left of Figure~\ref{fig:riemcone} it can be 
observed that as the initial curve $F(x_h^0(I))$ is homotopic to a point on 
$\mathcal M$, it shrinks to a point away from the apex.
On recalling Conjecture~5.1 in \cite{angenent}, due to Charles M. Elliott,
on the right of Figure~\ref{fig:riemcone} we also show a numerical experiment 
for a curve that is not homotopic to a point on $\mathcal M$. According to 
the conjecture, any such
curve should shrink to a point at the apex in finite time, and this is indeed
what we observe.
\begin{figure}
\center
\includegraphics[angle=-90,width=0.24\textwidth,align=t]{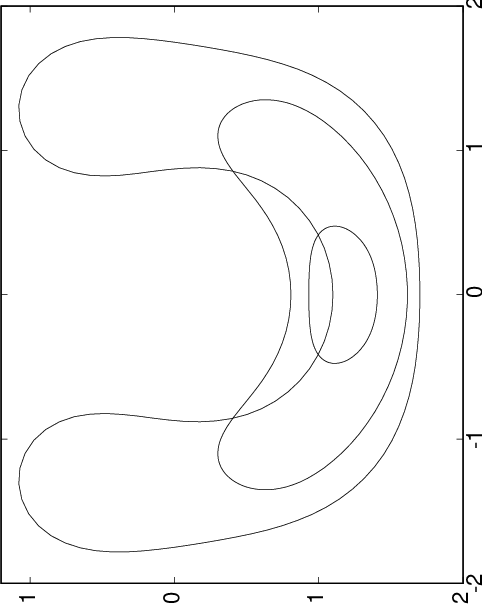}
\includegraphics[trim=80 20 70 20,clip,width=0.27\textwidth,align=t]{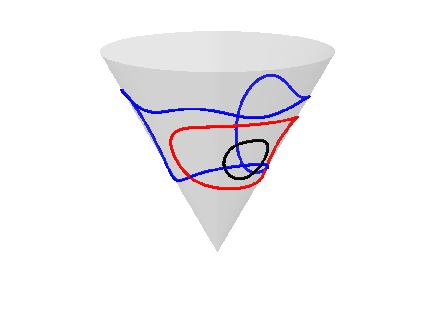}\quad\
\includegraphics[angle=-90,width=0.16\textwidth,align=t]{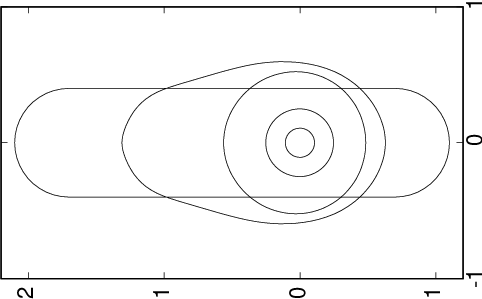}
\includegraphics[trim=80 20 70 20,clip,width=0.27\textwidth,align=t]{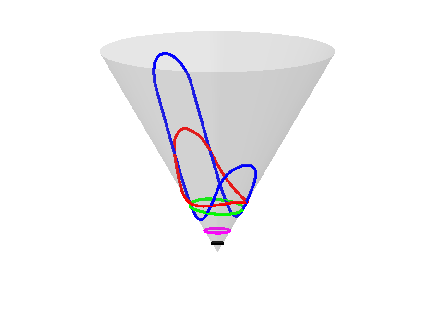}
\caption{Geodesic curvature flow on the cone defined by \eqref{eq:Fvarphi} 
with $\varphi(z) = \sqrt{3} |z|$.
We show the evolution of $x^m_h$ in $\Omega$, 
as well as of $F(x^m_h)$ on $\mathcal M$, 
at times $t=0, 1, 1.8$ (left) and $t=0, 0.2, 0.6, 1, 1.1$ (right).
} 
\label{fig:riemcone}
\end{figure}%

For the final set of numerical simulations, we model a surface with two 
mountains. Following \cite{WuT10}, we define
\begin{equation} \label{eq:varphi}
\varphi(z) = \lambda_1 \psi(|z|^2) + \lambda_2 \psi(|z- \tbinom20|^2),
\quad \lambda_1, \lambda_2 \in \bRgeq,
\qquad \text{where } \
\psi(s) = \begin{cases}
 e^{- \frac1{1 - s}} & s < 1,\\
 0 & s \geq 1,
\end{cases} 
\end{equation}
and let $\Omega = \bR^2$.
We show three evolutions for geodesic curvature flow on such surfaces
in Figures~\ref{fig:riemmount2small}, \ref{fig:riemmount2uneven}
\ref{fig:riemmount2stuck}.
In each case we start the evolution from an
equidistributed approximation of a circle of radius 2 in $\Omega$, 
centred at the origin.
In the first two simulations the curve manages to continuously decrease its
length in $\bR^3$, until it shrinks to a point. To achieve this in the second
example, the curve needs to ``climb up'' the higher mountain.
However, in the final example the two mountains are too steep, and so the curve
can no longer decrease its length by climbing higher. 
In fact, the curve approaches a steady state for the flow, that is,
a geodesic on $\mathcal M$, i.e.\ a curve with vanishing geodesic curvature.
The plot of the discrete energy in Figure~\ref{fig:riemmount2stuck} confirms
that the evolution is approaching a geodesic.
\begin{figure}
\center
\includegraphics[trim=80 50 70 50,clip,width=0.2\textwidth]{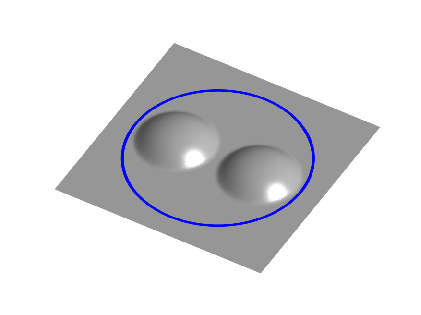}
\includegraphics[trim=80 50 70 50,clip,width=0.2\textwidth]{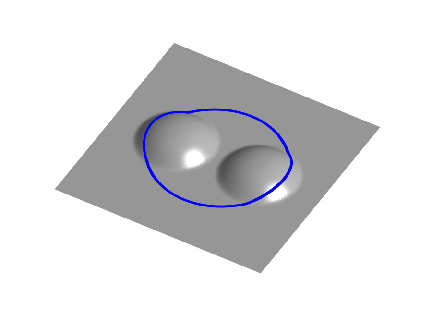}
\includegraphics[trim=80 50 70 50,clip,width=0.2\textwidth]{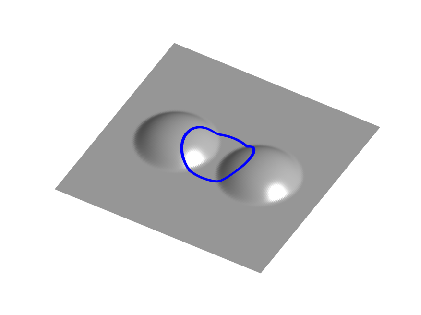}
\includegraphics[trim=80 50 70 50,clip,width=0.2\textwidth]{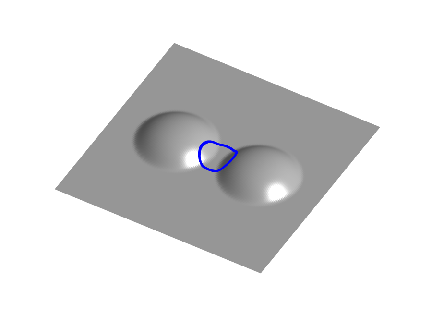}
\caption{Geodesic curvature flow on the graph defined by
\eqref{eq:varphi} with $\lambda_1=\lambda_2=1$.
We show the evolution of $F(x^m_h)$ on $\mathcal M$ at times $t=0, 1, 2, 2.2$.
} 
\label{fig:riemmount2small}
\end{figure}%
\begin{figure}
\center
\includegraphics[trim=100 50 90 50,clip,width=0.2\textwidth]{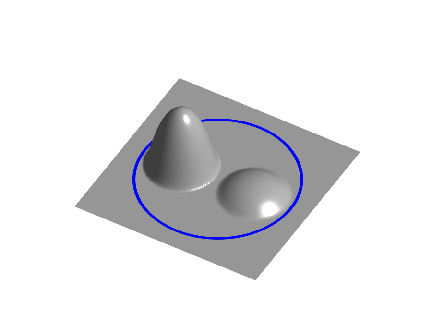}
\includegraphics[trim=100 50 90 50,clip,width=0.2\textwidth]{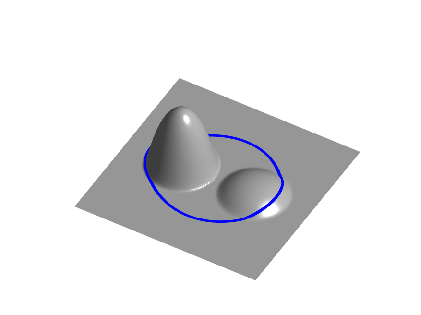}
\includegraphics[trim=100 50 90 50,clip,width=0.2\textwidth]{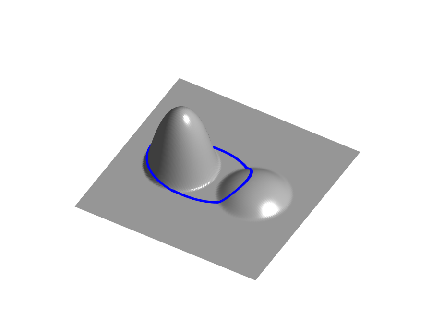}
\includegraphics[trim=100 50 90 50,clip,width=0.2\textwidth]{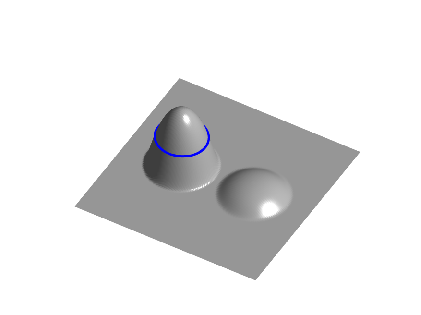}
\caption{Geodesic curvature flow on the graph defined by
\eqref{eq:varphi} with $(\lambda_1,\lambda_2)=(5,1)$.
We show the evolution of $F(x^m_h)$ on $\mathcal M$ at times $t=0, 1, 2, 4$.
} 
\label{fig:riemmount2uneven}
\end{figure}%
\begin{figure}
\center
\includegraphics[trim=100 50 90 50,clip,width=0.2\textwidth]{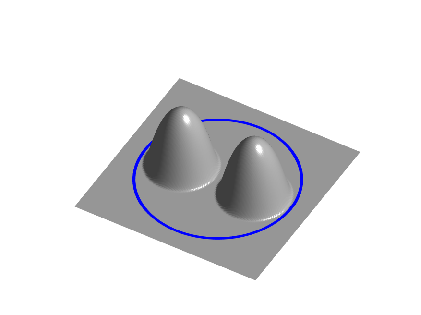}
\includegraphics[trim=100 50 90 50,clip,width=0.2\textwidth]{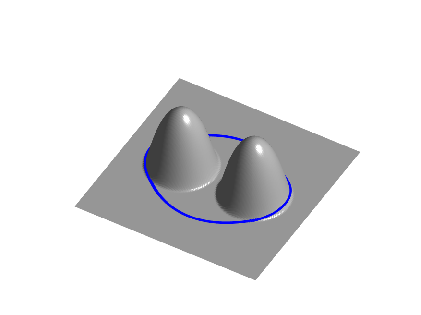}
\includegraphics[trim=100 50 90 50,clip,width=0.2\textwidth]{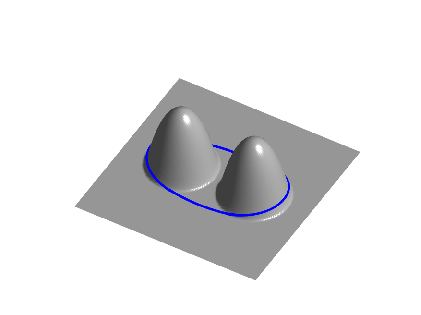}
\includegraphics[trim=100 50 90 50,clip,width=0.2\textwidth]{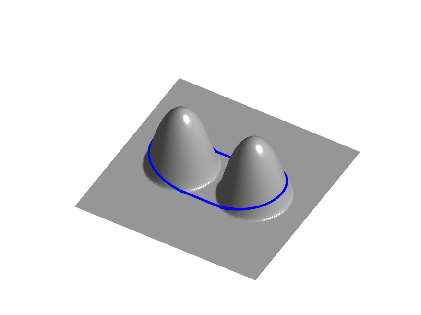} \\
\includegraphics[trim=110 70 100 80,clip,width=0.3\textwidth,align=t]{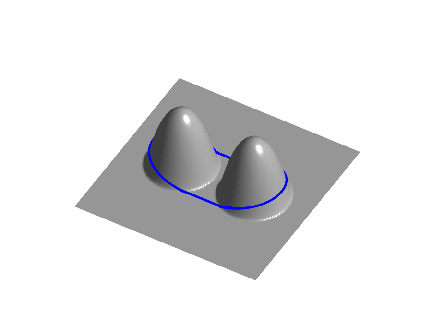} \quad
\includegraphics[angle=-90,width=0.35\textwidth]{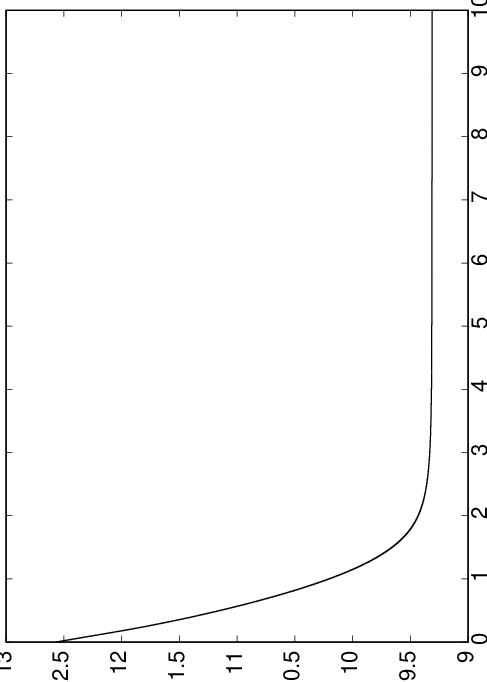}
\caption{Geodesic curvature flow on the graph defined by
\eqref{eq:varphi} with $\lambda_1=\lambda_2=5$.
We show the evolution of $F(x^m_h)$ on $\mathcal M$ at times $t=0, 1, 2, 4$.
Below we show a plot of $F(x^M_h)$ on $\mathcal M$, as well as a plot
of the discrete energy $\mathcal{E}^h(x_h^m)$ over time.
} 
\label{fig:riemmount2stuck}
\end{figure}%

\begin{appendix}
\renewcommand{\theequation}{\Alph{section}.\arabic{equation}}
\setcounter{equation}{0}
\section{First variation of the anisotropic energy} \label{sec:appA}

{\it Proof of Lemma~\ref{lem:firstvar}}.
Abbreviating $\tilde\gamma(z,p) = \altm(z)\gamma(z,p)$, 
$(z,p) \in \Omega \times \bR^2$, we temporarily write 
$\mathcal E$ in (\ref{eq:aniso}) as 
\begin{displaymath}
\mathcal E(\Gamma) = \int_\Gamma \tilde \gamma(\cdot,\nu) \dH1. 
\end{displaymath}
Let us fix a curve $\Gamma \subset \Omega$ and a smooth vector field
$V$ defined in an open neighbourhood of $\Gamma$. 
We infer from Corollary~4.3 in \cite{DoganN12} and (\ref{eq:phidd}) that 
the first variation of $\mathcal E(\Gamma)$ in the direction $V$ is
given by
\begin{align} \label{eq:firstvar1}
{\rm d} \mathcal E(\Gamma; V) & = 
\int_{\Gamma} \bigl( ( \tilde \gamma(\cdot,\nu) - \tilde \gamma_p(\cdot,\nu) 
\cdot \nu) \varkappa + \partial_\nu \tilde \gamma(\cdot,\nu) 
+ \mbox{div}_\Gamma \tilde \gamma_p(\cdot,\nu) 
+ \tilde \gamma_{pp}(\cdot,\nu) : \nabla_\Gamma \nu \bigr) \, V \cdot \nu 
\dH1 \nonumber \\
& = \int_{\Gamma} \bigl( \partial_\nu \tilde \gamma(\cdot,\nu) 
+ \mbox{div}_\Gamma \tilde \gamma_p(\cdot,\nu) 
+ \tilde \gamma_{pp}(\cdot,\nu) : \nabla_\Gamma \nu \bigr) \,  V \cdot \nu \dH1. 
\end{align}
Here we note that the differential operators 
$\partial_\nu f = f_{z_i} \nu_i$ and 
$\div_\Gamma f = f_{i,z_i} - f_{i, z_j} \nu_j \nu_i$ on $\Gamma$ 
only act on the first variable of functions defined in $\Omega \times \bR^2$. 
In addition, we observe that the Weingarten map $\nabla_\Gamma \nu$ is given by $\nabla_\Gamma \nu  = -\varkappa \tau \otimes\tau$.
We then calculate, on noting \eqref{eq:phidd}, that 
\begin{align*}
\partial_\nu \tilde \gamma(\cdot,\nu) & = 
\partial_\nu \altm \gamma(\cdot,\nu) + \altm \gamma_{z_i}(\cdot,\nu) \nu_i, \\
\mbox{div}_\Gamma \tilde \gamma_p(\cdot,\nu) & = 
\mbox{div}_\Gamma \bigl( \altm \gamma_p(\cdot,\nu) \bigr) \\ &
= \altm \bigl( \gamma_{p_i z_i}(\cdot,\nu) 
- \gamma_{p_i z_j}(\cdot,\nu) \nu_i \nu_j \bigr)
+ \bigl( \altm_{z_i} - \partial_\nu \altm \nu_i \bigr) \gamma_{p_i}(\cdot,\nu) 
\\ & 
= \altm \bigl( \gamma_{p_i z_i}(\cdot,\nu) - \gamma_{z_j}(\cdot,\nu) \nu_j 
\bigr) + \nabla \altm \cdot \gamma_p(\cdot,\nu) 
- \partial_\nu \altm \gamma(\cdot,\nu), \\
\tilde \gamma_{pp}(\cdot,\nu) : \nabla_\Gamma \nu & = 
- \altm \gamma_{pp}(\cdot,\nu) : \varkappa \tau \otimes \tau 
= -\altm \varkappa \gamma_{pp}(\cdot,\nu) \tau \cdot \tau.
\end{align*}
If we insert the above relations into (\ref{eq:firstvar1}) and recall (\ref{eq:nugamma}), we obtain
\begin{displaymath}
{\rm d} \mathcal E(\Gamma; V) 
= -\int_{\Gamma} \varkappa_\gamma \,  V \cdot \nu \,  \altm \dH1 
= - \int_{\Gamma} \varkappa_\gamma \,  V \cdot \nu_\gamma \,  \gamma(\cdot,\nu) 
\altm \dH1,
\end{displaymath}
which is \eqref{eq:firstvar2}.
\qedsymbol

\setcounter{equation}{0}
\section{Geodesic curve shortening flow in Riemannian manifolds} 
\label{sec:appB}

In this appendix we prove the claims formulated at the end of Example~\ref{ex:1}\ref{item:ex1c}.
Here we will make use of standard concepts in Riemannian geometry, and we 
refer the reader to the textbook \cite{Klingenberg78} %,doCarmo92}
for further details. \\
Let $F: \Omega \to \mathcal M$ be a local parameterization of 
a two-dimensional Riemannian manifold $(\mathcal M,g)$
and denote by $\{ \partial_1,\partial_2 \}$ the corresponding basis of the 
tangent space $T_{F(z)} \mathcal M$, for $z \in \Omega$. 
We also let $g_{ij}(z)=g_{F(z)}(\partial_i,\partial_j)$,
$G(z)=(g_{ij}(z))_{i,j=1}^2$, $(g^{ij}(z))_{i,j=1}^2=G^{-1}(z)$, 
$\gamma(z,p)= \sqrt{G^{-1}(z) p \cdot p}$ and $\altm(z)=\sqrt{\det G(z)}$,
for $z \in \Omega$ and $p \in \bR^2$, which induces the energy equivalence
\eqref{eq:Egamma}.
Let $\tilde \Gamma$ be a smooth curve in $\mathcal M$ with unit tangent 
$\tau_g$ and a unit normal $\nu_g$ such that $\{\tau_g, \nu_g \}$ 
is an orthonormal basis of the tangent space $T \mathcal M$,
i.e.\ $g(\tau_g,\tau_g) = g(\nu_g,\nu_g) = 1$ and $g(\tau_g,\nu_g)=0$.
Then the geodesic curvature $\varkappa_g$ of $\tilde \Gamma$ is defined by
\begin{equation} \label{eq:varkappag}
\varkappa_g = g (\Dds \tau_g, \nu_g)
\quad \text{on }\tilde\Gamma, 
\end{equation}
where $\Dds \tau_g$ is the covariant derivative of $\tau_g$.

\begin{lemma} \label{lem:appB0}
Let $\Gamma \subset \Omega$ be a smooth curve. Then the anisotropic curvature
of $\Gamma$ and the geodesic curvature of $\tilde\Gamma = F(\Gamma)$ coincide 
in the sense that $\varkappa_g \circ F = \varkappa_\gamma$ on $\Gamma$.
\end{lemma}
\begin{proof}
Let $\Gamma = x(I)$ for a parameterization 
$x: I \to \Omega$, so that $\tilde\Gamma = \tilde x(I)$ for 
$\tilde x = F \circ x$.
Denoting by $\tilde s$ the arclength of $\tilde x$,
we see that $\tau_g=\tilde x_{\tilde s}$ and 
$\nu_g= \frac{1}{\gamma(x,\nu)} g^{ij}(x) \nu_j \partial_i$
form an orthonormal basis of $T_{\tilde x} \mathcal M$.
Using the formula in \cite[Lemma~5.1.2]{Klingenberg78} we may write
\begin{equation} \label{eq:covder}
\Dds \tau_g=
 \Dds \tilde x_{\tilde s} = \bigl( x_{k,\tilde s \tilde s} + \Gamma^k_{ij}(x) x_{i,\tilde s} x_{j,\tilde s} \bigr) \partial_k,
\end{equation}
where $(\Gamma^k_{ij}(x))_{i,j,k=1}^2$ are the Christoffel symbols
of $\mathcal M$ at $F(x)$. Since $\partial_{\tilde s} = 
[G(x) x_\rho \cdot x_\rho]^{-\frac12} \partial_\rho$,  
\eqref{eq:varkappag}, \eqref{eq:covder} and \eqref{eq:tau} imply
\begin{align}
\varkappa_g \circ \tilde x & = 
g_{\tilde x} (\Dds \tau_g,\nu_g)= g_{kr}(x) \bigl( x_{k,\tilde s \tilde s} + \Gamma^k_{ij}(x) x_{i,\tilde s} x_{j,\tilde s} \bigr) 
\frac{1}{\gamma(x,\nu)} g^{lr}(x) \nu_l \nonumber  \\
& = \frac{1}{\gamma(x,\nu)} \bigl( x_{\tilde s \tilde s} \cdot \nu + \Gamma^k_{ij}(x) x_{i,\tilde s} x_{j, \tilde s} \nu_k \bigr) = 
\frac{1}{\gamma(x,\nu)} \frac{x_{\rho \rho} \cdot \nu + \Gamma^k_{ij}(x) x_{i,\rho} x_{j,\rho} \nu_k}{G(x) x_\rho \cdot x_\rho} \nonumber  \\
 & = \frac{1}{\gamma(x,\nu)} \frac{\varkappa + \Gamma^k_{ij}(x) \tau_i \tau_j \nu_k}{G(x) \tau \cdot \tau}. \label{eq:kappag}
\end{align}
On the other hand, on recalling
$\gamma(z,p)=\sqrt{G^{-1}(z) p \cdot p}$ and $\altm(z)= \sqrt{\det G(z)}$,
we observe that
\begin{subequations}
\begin{align}
\gamma_p( z, p) &= \frac{G^{-1}( z)  p}{\gamma( z, p)}, \quad
\gamma_{pp}( z, p) = \frac{G^{-1}( z)}{\gamma( z, p)} 
- \frac{ G^{-1}( z)  p \otimes G^{-1}( z)  p}{\gamma^3( z, p)}, 
\label{eq:riemphipp} \\
\gamma_{p z_j}( z, p) & = 
 \frac{(G^{-1})_{z_j}( z)  p}{\gamma( z, p)} 
-\tfrac12 \frac{(G^{-1})_{z_j}( z)  p \cdot  p}{\gamma^3( z, p)}  
G^{-1}( z)  p , \label{eq:riemphipz} \\
\altm_{z_j}(z) &= \tfrac12 \tr(G^{-1}(z) G_{z_j}(z)) \altm(z). \label{eq:riemaz}
\end{align}
\end{subequations}
We infer from \eqref{eq:riemphipp} that
\begin{equation} \label{eq:riemkappa1}
\gamma_{pp}( x,\nu) \tau \cdot \tau  
= \frac{1}{\gamma^3( x,\nu)} \left[(G^{-1}( x) \nu \cdot \nu) (G^{-1}( x) \tau \cdot \tau) - (G^{-1}(x) \nu \cdot \tau)^2 \right] 
= \frac{\det G^{-1}(x)}{\gamma^3( x,\nu)}. 
\end{equation}
For notational convenience, we drop the dependences %of $G=(g_{ij})_{i,j=1}^2$ 
on $x$ from now on. It is well-known that
\begin{equation}  \label{eq:christoffel}
g_{kl,z_i} = g_{kr} \Gamma^r_{il} + g_{lr} \Gamma^r_{ik}, \quad i,k,l=1,2.
\end{equation}
Combining \eqref{eq:christoffel}
with the relation $(G^{-1})_{z_i} = - G^{-1} G_{z_i} G^{-1}$, we find that
\[
[ (G^{-1})_{z_i} \nu]_j  =  - g^{jk} g_{kl,z_i} g^{lm} \nu_m = - g^{jk} g^{lm} \bigl( g_{kr} \Gamma^r_{il} + g_{lr} \Gamma^r_{ik} \bigr) \nu_m 
= - g^{lm} \Gamma^j_{il} \nu_m - g^{jk} \Gamma^m_{ik} \nu_m,
\]
as well as
\[
(G^{-1})_{z_i} \nu \cdot \nu =  [ (G^{-1})_{z_i} \nu]_j \nu_j = - \bigl( g^{lm} \Gamma^j_{il} \nu_m + g^{jk} \Gamma^m_{ik} \nu_m \bigr) \nu_j = - 2 g^{jk} \Gamma^m_{ik} \nu_m \nu_j.
\]
If we insert the above relations into \eqref{eq:riemphipz}, we obtain
\[
\gamma_{p_j z_j}(\cdot,\nu)= - \frac{ g^{lm} \Gamma^j_{jl} \nu_m + g^{jk} \Gamma^m_{jk} \nu_m}{\gamma(\cdot,\nu)} + \frac{g^{lk} \Gamma^m_{jk} \nu_m \nu_l}{\gamma^3(\cdot,\nu)}
g^{jr} \nu_r.
\]
Next we infer with the help of \eqref{eq:riemaz} that
\[
\frac{\altm_{z_j}}{\altm} \gamma_{p_j}(\cdot,\nu) = \tfrac12 \frac{g^{kl} g_{kl,z_j}  g^{jr} \nu_r}{\gamma(\cdot,\nu)} = 
\tfrac12 \frac{g^{kl} g^{jr} \bigl( g_{kr} \Gamma^r_{jl} + g_{lr} \Gamma^r_{jk} \bigr) \nu_r}{\gamma(\cdot,\nu)}
=  \frac{g^{jr} \Gamma^k_{jk} \nu_r}{\gamma(\cdot,\nu)}.
\]
As a result,
\[
\gamma_{p_j z_j}(\cdot,\nu) + \frac{\nabla \altm}{\altm} \cdot \gamma_p(\cdot,\nu) = - \frac{ g^{jk} \Gamma^m_{jk} \nu_m}{\gamma(\cdot,\nu)}
+ \frac{ g^{lk} g^{jr} \Gamma^m_{jk} \nu_m \nu_l \nu_r}{\gamma^3(\cdot,\nu)} 
= \frac{1}{\gamma^3(\cdot,\nu)} \bigl( g^{lk} g^{jr} - g^{jk} g^{lr} \bigr) \Gamma^m_{jk} \nu_m \nu_l \nu_r.
\]
Clearly,
\[
g^{lk} g^{jr} - g^{jk} g^{lr} = 
\begin{cases}
(\det  G)^{-1} & (l,k,j,r)=(1,1,2,2), (2,2,1,1), \\
-(\det G)^{-1} & (l,k,j,r)=(1,2,2,1),(2,1,1,2), \\
0 & \mbox{ otherwise},
\end{cases}
\]
so that
\begin{displaymath} 
\gamma_{p_j z_j}(\cdot,\nu) + \frac{\nabla \altm}{\altm} \cdot \gamma_p(\cdot,\nu) = - \frac{\det G^{-1}}{\gamma^3(\cdot,\nu)} \bigl(
\Gamma^m_{22} \nu_m \nu_1^2 + \Gamma^m_{11} \nu_m \nu_2^2 - 2 \Gamma^m_{12} \nu_m \nu_1 \nu_2 \bigr) 
 =  - \frac{\det  G^{-1} \, \Gamma^m_{kl} \nu_m \tau_k \tau_l}{\gamma^3(\cdot,\nu)},
\end{displaymath}
since $\tau=-\nu^\perp$.
Combining this relation with \eqref{eq:nugamma}, \eqref{eq:riemkappa1}, 
\eqref{eq:kappag} and the fact that
$\gamma^2(\cdot,\nu)=G^{-1} \nu \cdot\nu = (\det G^{-1}) G \tau \cdot\tau$,
we finally obtain that 
\begin{equation*} %\label{eq:transform}
\varkappa_\gamma \circ x
= \frac{\det G^{-1}}{\gamma^3(\cdot,\nu)} \bigl(\varkappa 
+ \Gamma^m_{kl} \tau_k \tau_l \nu_m \bigr) 
= \frac{1}{\gamma(\cdot,\nu)} 
\frac{\varkappa + \Gamma^m_{kl} \tau_k \tau_l \nu_m}{G \tau \cdot \tau} 
= \varkappa_g \circ \tilde x = (\varkappa_g \circ F) \circ x
\quad \text{in } I,
\end{equation*}
as claimed.
\end{proof}

A family of curves $(\tilde\Gamma(t))_{t \in [0,T]}$ in $\mathcal M$,
is said to evolve by geodesic curvature flow if
\begin{equation} \label{eq:gcsf}
\mathcal V_g = \varkappa_g
\quad \text{on } \tilde\Gamma(t),
\end{equation}
where $\mathcal V_g$ is the normal velocity in the direction of the unit 
normal $\nu_g$ from definition \eqref{eq:varkappag}, i.e.\
$\mathcal V_g = g(\tilde x_t \circ \tilde x^{-1}, \nu_g)$
with $\tilde x : I\times[0,T] \to \Omega$
being a parameterization of $(\tilde\Gamma(t))_{t \in [0,T]}$.

\begin{lemma} \label{lem:appB}
Let $(\Gamma(t))_{t \in [0,T]}$ be a smooth family of curves in $\Omega$.
Then anisotropic curve shortening flow for $(\Gamma(t))_{t \in [0,T]}$ in
$\Omega$, \eqref{eq:acsf}, is equivalent to geodesic curvature flow
for $(F(\Gamma(t)))_{t \in [0,T]}$ in $\mathcal M$, 
\eqref{eq:gcsf}.
\end{lemma}
\begin{proof}
Similarly to the proof of Lemma~\ref{lem:appB0}, we assume that
$(\Gamma(t))_{t \in [0,T]}$ is parameterized by
$x : I \times [0,T] \to \Omega$, so that $\tilde x = F \circ x$ parameterizes
$(F(\Gamma(t)))_{t \in [0,T]}$. Let
$\nu_g= \frac{1}{\gamma(x,\nu)} g^{ij}(x) \nu_j \partial_i$.
Then it follows from $\tilde x_t = x_{k,t} \partial_k$ that
\begin{align} \label{eq:VgVgamma}
(\mathcal V_g \circ F ) \circ x & = 
\mathcal V_g \circ \tilde x = 
g_{\tilde x}(\tilde x_t, \tilde \nu) = 
\frac{1}{\gamma(x,\nu)} g^{ij}(x) \nu_j x_{k,t} \, 
g_{\tilde x}(\partial_k,\partial_i)
= \frac{1}{\gamma(x,\nu)} g^{ij}(x) g_{ki}(x) \nu_j x_{k,t} \nonumber \\ &
= \frac{1}{\gamma(x,\nu)} x_t \cdot \nu 
= \mathcal V_\gamma \circ x \quad \text{in } I \times (0,T].
\end{align}
Combining \eqref{eq:VgVgamma} and Lemma~\ref{lem:appB0} yields the desired
result.
\end{proof}

\end{appendix}

\def\soft#1{\leavevmode\setbox0=\hbox{h}\dimen7=\ht0\advance \dimen7
  by-1ex\relax\if t#1\relax\rlap{\raise.6\dimen7
  \hbox{\kern.3ex\char'47}}#1\relax\else\if T#1\relax
  \rlap{\raise.5\dimen7\hbox{\kern1.3ex\char'47}}#1\relax \else\if
  d#1\relax\rlap{\raise.5\dimen7\hbox{\kern.9ex \char'47}}#1\relax\else\if
  D#1\relax\rlap{\raise.5\dimen7 \hbox{\kern1.4ex\char'47}}#1\relax\else\if
  l#1\relax \rlap{\raise.5\dimen7\hbox{\kern.4ex\char'47}}#1\relax \else\if
  L#1\relax\rlap{\raise.5\dimen7\hbox{\kern.7ex
  \char'47}}#1\relax\else\message{accent \string\soft \space #1 not
  defined!}#1\relax\fi\fi\fi\fi\fi\fi}

\end{document}